\documentclass[12pt,english]{article}  

\usepackage[width=17.5cm,height=23cm]{geometry}
\usepackage{amsmath,amssymb,amsthm,latexsym}
\usepackage{stmaryrd}
\usepackage{tikz}
\usepackage{tikz-cd}
\usepackage{cleveref}

\newcommand{\A}{\mathbb{A}}
\newcommand{\Z}{\mathbb{Z}}
\newcommand{\N}{\mathbb{N}}
\newcommand{\NN}{\mathbb{N}}
\newcommand{\Q}{\mathbb{Q}}
\newcommand{\FF}{\mathbb{F}}

\newcommand{\SH}{\mathbf{SH}}

\newcommand{\unit}{\mathbf{1}}

\newcommand{\MSL}{\mathsf{MSL}}

\newcommand{\KGL}{\mathsf{KGL}}
\newcommand{\kgl}{\mathsf{kgl}}
\newcommand{\KQ}{\mathsf{KQ}}
\newcommand{\kq}{\mathsf{kq}}
\newcommand{\kw}{\mathsf{kw}}
\newcommand{\KW}{\mathsf{KW}}
\newcommand{\C}{\mathsf{C}}
\newcommand{\D}{\mathsf{D}}
\newcommand{\E}{\mathsf{E}}

\newcommand{\Witt}{\mathbf{W}}

\newcommand{\vs}{{\mathsf{vs}}}
\newcommand{\f}{\mathsf{f}}
\newcommand{\vf}{{\mathsf{vf}}}
\newcommand{\pis}{\underline{\pi}}

\renewcommand{\P}{\mathbb{P}}

\newcommand{\SL}{\mathrm{SL}}
\newcommand{\Tan}{\mathrm{Tan}}
\newcommand{\Id}{\mathrm{Id}}

\newcommand{\cell}{\mathrm{cell}}

\newcommand{\Gm}{\mathbb{G}_m}

\newcommand{\CC}{\mathbb{C}}
\newcommand{\HZ}{\mathsf{H}\Z}

\newcommand{\HA}{\mathsf{H}A}

\newcommand{\Ab}{\mathbf{Ab}}
\newcommand{\Top}{\mathrm{Top}}
\newcommand{\Sm}{\operatorname{Sm}}

\newcommand{\colim}{\operatornamewithlimits{colim}}
\newcommand{\holim}{\operatornamewithlimits{holim}}

\newcommand{\tensor}{\otimes}
\newcommand{\Spec}{\operatorname{Spec}}

\newcommand{\Sq}{\operatorname{Sq}}

\newcommand{\pr}{\operatorname{pr}}

\newcommand{\inc}{\operatorname{inc}}
\newcommand{\ret}{\operatorname{ret}}
\newcommand{\s}{\mathsf{s}}

\newcommand{\KMW}{\mathbf{K}^\mathrm{MW}}
\newcommand{\KMil}{\mathbf{K}^\mathrm{M}}
\newcommand{\KWitt}{\mathbf{K}^\mathrm{W}}

\newcommand{\Forg}{\mathsf{Forg}}
\newcommand{\Hyper}{\mathsf{Hyp}}

\newcommand{\hyp}{\mathsf{h}}

\let\iso=\cong
\let\smash=\wedge
\let\directsum=\oplus
\let\minus=\smallsetminus

\author{Oliver R\"ondigs}
\title{The zeroth stable homotopy groups of motivic spheres over the integers}
\date{October 2, 2026}

\newtheorem{theorem}{Theorem}[section]
\newtheorem*{theorem*}{Theorem}
\newtheorem{lemma}[theorem]{Lemma}
\newtheorem{proposition}[theorem]{Proposition}
\newtheorem*{proposition*}{Proposition}
\newtheorem{corollary}[theorem]{Corollary}
\newtheorem*{corollary*}{Corollary}

\theoremstyle{definition}
\newtheorem{defn}[theorem]{Definition}
\newtheorem{defn*}{Definition}
\newtheorem{remark}[theorem]{Remark}
\newtheorem{remark*}{Remark}
\newtheorem{example}[theorem]{Example}

\numberwithin{equation}{section}

\begin{document}
\maketitle
\begin{abstract}
  The main result determines the
  zeroth integral Milnor-Witt stem of the motivic sphere
  spectrum in the Morel-Voevodsky motivic
  stable homotopy category of the integers.
  The component in weight zero is the Grothendieck-Witt ring
  of nondegenerate symmetric bilinear forms over $\Z$.
  Along the way, cellularity of connective Witt theory,
  as well as an absolute purity result
  for the $\eta$-inverted motivic sphere spectrum, is
  established over Dedekind domains
  of mixed characteristic.
\end{abstract}

\section{Introduction}
\label{section:introduction}

A key foundational result in the Morel-Voevodsky motivic homotopy theory
is Morel's Theorem identifying the endomorphism ring of the motivic sphere
spectrum $\unit_F$ over a field $F$ in the motivic stable homotopy category
$\SH(F)$ with the Grothendieck-Witt ring of symmetric bilinear forms
over $F$ \cite{morel.newton}. Since the motivic sphere spectrum is the
unit with respect to the symmetric monoidal structure on $\SH(F)$
given by the smash product, the
homotopy groups of every
motivic spectrum in $\SH(F)$ receive a natural action by the
Grothendieck-Witt ring of symmetric bilinear forms over $F$.
As another application,
certain enumerative geometry problems classically answered with
an integer $n$ now admit refinements as the class of a virtual
symmetric bilinear form with the rank $n$ \cite{levine.refined}.

The surprise appearance of symmetric bilinear forms actually occurs
in weight zero of an integer-graded ring of endomorphisms shifted
by smash powers of the Tate circle $\Gm$. Morel succeeded in
identifying not only these shifted endomorphisms as the so-called
Milnor-Witt $K$-theory of $F$ (an extension of Milnor's classical
symbolic $K$-theory), but even enhanced this endomorphism identification 
to the level of
Nisnevich sheaves on the site of smooth $F$-schemes. Moreover, he
lifted this result to the unstable $\A^1$-homotopy category
of a perfect field \cite{morel.at}.

Extensions of Morel's Theorem to other bases than fields exist, but
tend to come with restrictions and modifications. For example,
the main result of \cite{bachmann-ostvaer.models} transfers the weight zero
case of Morel's Theorem to $2$-regular number rings in which $2$ is inverted,
after localizing at $2$.
This is extended to all weights in work by Figini and {\O}stv{\ae}r, after
$2$-completion or $2$-localization. 
More general Dedekind schemes are discussed in Figini's dissertation.
Over a discrete valuation ring with residue field of odd characteristic $p$
and fraction field of characteristic $0$, Morel's Theorem holds up to
weight $1$ after inverting $p$ by \cite[Theorem 7.13]{rs.lecture}.

The main result proved here as one case of Theorem~\ref{thm:pi0-sphere-z}
is to identify the so-called zeroth Milnor-Witt stem $\pi_{0+(\star)}\unit_\Z$
of the motivic sphere
spectrum over the initial ring of integers with the naive Milnor-Witt
$K$-theory $\KMW_{-\star}(\Z)$
of the integers. In a more down to earth form, this
integer-graded collection of abelian groups can be
pictured as follows:
\[\begin{tikzcd}
  \dotsm \ar[r,"-2",shift left] & \Z\{[-1]^2\} \ar[r,"-2",shift left] \ar[l,"1",shift left]
  & \Z\{[-1]\} \ar[r,"{(-2,1)}",shift left] \ar[l,"1",shift left] &
  \Z\{1\}\oplus \Z\{\hyp\} \ar[r,"\pr_1",shift left] \ar[l,"\pr_1",shift left] &
  \Z\{\eta\} \ar[r,"1",shift left] \ar[l,"{(-2,1)}",shift left] & \Z\{\eta^2\} \ar[r,"1",shift left] \ar[l,"-2",shift left] & \dotsc \ar[l,"-2",shift left]  \\
  \star & -2 & -1 & 0 & 1 & 2 & 
\end{tikzcd}
\]
Here $\hyp=2+\eta[-1]$ represents the hyperbolic plane, arrows pointing to the
left denote multiplication with the unit $[-1]$, and arrows pointing to the
right denote multiplication with the first algebraic Hopf map $\eta$.
Moreover, this is the first nontrivial Milnor-Witt stem of $\unit_\Z$
in the sense that $\pi_{s+(\star)}\unit_\Z=0$ for $s<0$.
Admirably, all abelian groups appearing are finitely generated free,
which is the way it should be.

While Theorem~\ref{thm:pi0-sphere-z} achieves a similar statement
for all the local rings $\Z_{(p)}$ (at least if $p>5$), the disappointment of
course is the focus on just the homotopy groups instead of the homotopy
sheaves, and then only over very few number rings.
More complicated number rings like $\Z[\sqrt{-5}]$ present additional
challenges in that
$\pi_{-1-(1)}\unit_{\Z[\sqrt{-5}]}\iso \mathrm{Pic}(\Z[\sqrt{-5}])\iso \Z/2$.
Multiplying the nonzero element in that group with elements from
$\pi_{1+(\star)}\unit_{\Z[\sqrt{-5}]}$
such as the second algebraic Hopf map $\nu$ and the first topological
Hopf map $\eta_\Top$ potentially produces elements in
$\pi_{0+(\star)}\unit_{\Z[\sqrt{-5}]}$ which are not captured
by the naive Milnor-Witt ring of $\Z[\sqrt{-5}]$.
Figini's dissertation studies such situations.

In the course of this computation, several results are obtained which
might be of independent interest, such as the cellularity of
connective Witt theory over Dedekind domains
(Proposition~\ref{prop:kw-cellular}), a purity
result for the $\eta$-inverted sphere $\unit[\eta^{-1}]$ over
discrete valuation rings of mixed characteristic
(Corollary~\ref{cor:absolute-purity-etainv}), and a few
computations over fields of characteristic $2$.
A crucial device is the extension of
periodic hermitian $K$-theory to an absolute motivic ring spectrum
with plenty of pleasant properties \cite{chn}, among them
cellularity \cite{kumar-roendigs}.
Lifting these properties to its very effective cover, at least to some
extent, requires additional efforts, and is achieved here only partially.

\subsubsection*{Notation}
Let $S$ be a qcqs (quasi-compact and quasi-separated) base scheme.
The case of most relevance occurs if $S$ is (essentially smooth over)
a Dedekind scheme.
A Dedekind scheme 
refers to a finite disjoint union of spectra of
Dedekind domains or fields;
that is, a regular noetherian scheme of Krull dimension at most $1$.

Given integers $s,w\in \Z$, the motivic sphere
$S^{s+(w)}=S^{s+w,w}$ is the $s$-fold simplicial suspension of the
$w$-fold smash product of the Tate circle $\Gm$ with itself.
It could be called the motivic sphere of connectivity $s$ and weight $w$,
or of topological degree $s+w$ and weight $w$.
The corresponding suspension functor on the stable motivic homotopy category $\SH(S)$ is denoted 
by $\Sigma^{s+(w)}=\Sigma^{s+w,w}$.
To any motivic spectrum $\E\in\SH(S)$ and $s,w\in \Z$ one
associates the motivic homotopy group 
\[
\pi_{s+(w)}\E:=
\pi_{s+w,w}\E:=
[S^{s+(w)},\E]_{\SH(S)} =: [S^{s+(w)},\E]_S
\]
where the last two notations denote abelian groups of morphisms in $\SH(S)$.
These may be assembled to the $s$-th Milnor-Witt stem of $\E$ as
follows:
\[ \pi_{s+(\star)}\E:=\pi_s\E := \bigoplus_{w\in \Z}\pi_{s+(w)}\E\]
The presheaf $\Sm_S\to\Ab,\, X\mapsto [\Sigma^{s+(w)}X_+,\E]_S$
of homotopy groups sheafifies to a Nisnevich sheaf $\pis_{s+(w)}\E$
on $\Sm_S$; similar conventions as depicted for homotopy groups apply
to these homotopy sheaves. If $S$ is local henselian, the homotopy groups
of $\E\in \SH(S)$ are the global sections of the homotopy sheaves of $\E$.

Besides the motivic sphere spectrum $\unit_S$ (which is the unit of a
closed symmetric monoidal structure on $\SH(S)$ denoted by $\smash$),
the motivic Eilenberg-MacLane spectra $\HZ_S$ and $\HA_S$
(for the ring of integers
and an abelian group $A$), the periodic algebraic and hermitian $K$-theory
spectra $\KGL_S$ and $\KQ_S$, and their very effective covers appear in
the following. The subscript denoting the base scheme may be left out
occasionally. As an abbreviation, $H^{s+w,w}(S):=\pi_{-s-w,-w}\HZ_S$ and 
$h^{s+w,w}(S):=\pi_{-s-w,-w}\HZ_S/2$ occur.

\subsection*{Acknowledgements}

The work presented here was announced at the conference
``Motivic homotopy theory'' held within the DFG-Sonderforschungsbereich 1085
``Higher Invariants'' at Regensburg University; I thank sincerely
for this invitation and the support. Its documentation took longer
than expected, which is my sole responsibility.
I thank Alexey Ananyevskiy, Andrea Figini, Markus Spitzweck,
and Paul Arne {\O}stv{\ae}r
for helpful discussions, and Marc Hoyois for a helpful
comment. The work was
partially funded by
zukunft.niedersachsen, the joint science funding program of the
Lower Saxony Ministry of Science and Culture and the Volkswagen
Foundation, and by the DFG-Sachbeihilfe with project number
539085450. Finally, I would like to thank the Simons Laufer
Institute for Mathematical Sciences for excellent working conditions,
allowing me to finish this documentation while participating as a Senior
Clay Scholar during the Fall 2026 semester. In particular
this work was partially supported by
the Clay Mathematics Institute and the National Science Foundation under
Grant No.~DMS-2424139.

\section{Milnor-Witt $K$-theory}\label{sec:milnor-witt-k}

\begin{defn}\label{def:milnor-witt-ktheory}
  Let $R$ be a commutative ring.
  The \textit{(naive) Milnor-Witt} $K$-\textit{theory} $\KMW_\star(R)$
  of $R$ has the Hopkins-Morel
  presentation as the associative graded
  ring (with unit) whose generators are the multiplicative units $[u],u\in R^\times$
  in degree 1, and a generator $\eta$ in degree $-1$, subject
  to four relations:
  \begin{description}
  \item[Steinberg] $[u]\cdot [v] = 0$ if $u+v=1$.
  \item[Logarithm] $[uv] = [u]+ [v] +\eta\cdot [u]\cdot [v]$
  \item[Commutativity] $[u]\cdot \eta = \eta\cdot [u]$
  \item[Hyperbolic Plane] $\hyp\cdot \eta = 0$, where $\hyp:=2+\eta[-1]$
  \end{description}
  The definition extends to schemes $S$ via their ring of global sections.
\end{defn}

To work with this ring, a few of its properties will be used. These
are due to Morel in the case of
fields \cite[Lemma 3.5]{morel.at} and generalized to the setting here in
\cite[Lemma 4.13]{schlichting.euler}. To conveniently state them, let
$\langle u \rangle := 1+\eta\cdot [u]$.

\begin{lemma}\label{lem:milnor-witt-ktheory}
  Let $S$ be a scheme, and $u,v$ units in the global sections of $S$.
  Then $\langle uv\rangle = \langle u \rangle \cdot \langle v\rangle$.
  Moreover, $[1]=0$ and $\langle 1\rangle =1$.
  In particular, $\langle u \rangle$ is a unit in $\KMW(S)$ for every unit $u$ in $S$.
\end{lemma}

\begin{example}\label{ex:kmw-integers}
  The naive Milnor-Witt $K$-theory of $\Z$ has the following form;
  see \cite[Remark 2.2.8]{deglise.kmw-notes} for a slightly different
  description.
  For a nonzero integer $n$, the group $\KMW_n(\Z)$ is infinite cyclic, generated
  by $[-1]^n$ if $n$ is positive, and by $\eta^{-n}$ if $n$ is negative.
  The ring $\KMW_0(\Z)$ is isomorphic to the Grothendieck-Witt ring of
  non-degenerate symmetric bilinear forms over $\Z$. In other words, it is isomorphic to
  $\Z\{1\}\directsum \Z\{\eta\cdot [-1]\}$, with
  $(\eta\cdot [-1])^2= -2 (\eta\cdot [-1])$ by the Commutativity and Logarithm (or
  Hyperbolic) relations.
  The remaining multiplicative structure is dictated by the effect of multiplying
  with $[-1]$ and $\eta$, which is given by the Logarithm and Hyperbolic relations.
  In particular, for every integer $n$, the abelian group $\KMW_n(\Z)$ is free
  finitely generated, and in particular torsion-free.
  To see this, observe that since $[1]=0$ by Lemma~\ref{lem:milnor-witt-ktheory},
  the only relevant unit is $-1$.
  Hence $\KMW(\Z)$ is a quotient of the (commutative!) polynomial
  ring over $\Z$ in two variables, $[-1]$ of degree $1$ and $\eta$ of degree $-1$.
  The Steinberg relation is irrelevant.
  The Logarithm relation reduces to $\eta\cdot [-1]^2 = -2[-1]$, the Hyperbolic
  relation says $\eta^2\cdot [-1] = -2\eta$.
  Hence the monomials reduce as follows for $k,\ell\in \NN$:
  \[ [-1]^k\cdot \eta^\ell =
    \begin{cases}
      (-2)^{\ell}[-1]^{k-\ell} & k>\ell \\ (-2)^{k-1}[-1]\cdot \eta & k = \ell >0 \\
      (-2)^k \eta^{\ell -k} & k < \ell 
    \end{cases}\]
  The aforementioned structure follows.
\end{example}

Two quotient rings of $\KMW(S)$ are of particular relevance, the naive Milnor $K$-theory
$\KMil(S)$ obtained as the quotient by the two-sided ideal generated by $\eta$,
and the naive Witt $K$-theory $\KWitt(S)$
obtained as the quotient by the two-sided ideal generated by $\hyp$.
By definition, $\KMil(S)$ is concentrated in nonnegative degrees, and consists
of the group $\Z$ in degree $0$.

\begin{example}\label{ex:kmil-kwitt-integers}
  The naive Milnor $K$-theory $\KMil(\Z)$ has the property that
  $\KMil_n(\Z)$ is the group with two elements for every positive $n$, generated
  by $[-1]^n$. The Logarithm relation implies $2[-1]=0\in \KMil_1(\Z)$.
  The naive Witt $K$-theory of $\Z$ satisfies
  $\KWitt_n(\Z)\iso \Z$ for every $n$.
  Multiplication with $\eta$ on $\KMW(\Z)$ induces a short exact sequence
  \[ 0 \to \KWitt_{n+1}(\Z) \to \KMW_n(\Z) \to \KMil_n(\Z) \to 0 \]
  for every integer $n$. 
\end{example}

\begin{example}\label{ex-kmil-2-local-integers}
  Milnor defined $\KMil(F)$ exclusively for fields in \cite{milnor.inv}.
  While its definition makes sense for general commutative rings,
  some properties of the resulting functor are less desireable,
  even after Zariski sheafification. One problem concerns local rings
  with very small residue fields. For example, there are no
  Steinberg relations in $\KMil(\Z_{(2)})$.
  Using the standard symbol notation with curly instead
  of right-angled brackets, the element
  $\{27,13\}\in \KMil_2(\Z_{(2)})=\Z_{(2)}^\times \tensor_\Z\Z_{(2)}^\times$
  is nonzero, but its image in $\KMil_2(\Q)$ is $0$, because
  \[ \{27,13\} = \{27,\frac{-1}{2}\}+\{27,-26\} = -\{3^3,-2\}
    =-3\{3,-2\} = 0\]
  as indicated in \cite[Remark 8.7]{vdkallen}.
  This ruins the
  desired injectivity, which holds
  for $\KMil_2(\Z_{(p)})\to \KMil_2(\Q)$ if $p>5$
  by loc.~cit. and \cite{dennis-stein}.
  An improved variant of Milnor $K$-theory is given in \cite{kerz.improved}.
\end{example}

Abusing notation, a unit $u$ in the global sections of $S$ defines an
unstable map of
spheres also denoted $[u]\colon S^0\to S^{(1)}$.
Let $\eta\colon S^{1+(2)}\simeq \A^2\minus \{0\} \to \P^1\simeq S^{1+(1)}$
also denote the unstable algebraic Hopf map.
With the standard convention of not notationally distinguishing between
unstable maps and their resulting maps of $\P^1$-suspension spectra,
one obtains $[u]\in \pi_{0-(1)}\unit_S$ and $\eta\in \pi_{0+(1)}\unit_S$.
Work of Morel \cite{morel.newton} and Hu-Kriz \cite{hu-kriz.steinberg},
extended in \cite{druzhinin.milnor-witt}, supplies the following statement.

\begin{theorem}\label{thm:milnor-witt-ktheory}
  The canonical map defines a ``degree inverting'' ring 
  homomorphism $\KMW_\star(S)\to \pi_{0-(\star)}\unit_S$.
\end{theorem}

To show that this homomorphism is an isomorphism for $S=\Spec(\Z)$,
the corresponding result for very effective hermitian $K$-theory will be
proven first. This in turn requires a brief reminder on connectivity
and effectivity notions in the motivic stable homotopy category.

\section{Connectivity, effectivity, and cellularity}
\label{sec:conn-effect}

The setting discussed in
\cite[Appendix B]{bachmann-hoyois.norms} supplies several filtrations
of interest on $\SH(S)$. Given a subset $I\subset \Z\times \Z$, let
$\SH(S)_{\geq I}$ be the full subcategory of $\SH(S)$ generated
under colimits and
extensions by the set $\{\Sigma^{s+(w)}X_+\colon X\in \Sm_S,(s,w)\in I\}$.
Its inclusion $\inc(S)_{\geq I}$ admits a right adjoint $\ret(S)_{\geq I}$.
The inclusion $\inc(S)_{< I}\colon \SH(S)_{< I}\hookrightarrow \SH(S)$ of the
corresponding right orthogonal full subcategory also admits a
right adjoint $\ret(S)_{< I}$. 
By \cite[Prop.~B.3]{bachmann-hoyois.norms}, 
$\E\in \SH(S)_{\geq I}$ if and only if $i_s^\ast\E\in \SH(F_s)_{\geq I}$ for
every point $s\in S$ having residue field inclusion $i_s\colon \Spec(F_s)\to S$.
Here $S$ is locally of finite Krull dimension.

Let $n\in \N$. A motivic spectrum in $\SH(S)_{\geq \Z\times \{n\}}$ is
called $n$-\textit{effective} (or simply \textit{effective} if $n=0$).
The composition $\inc(S)_{\Z\times \{n\}}\circ \ret(S)_{\Z\times \{n\}} =: \f(S)_n$,
called the $n$-effective cover, gives rise to Voevodsky's slice filtration
\[ \dotsc \to \f(S)_{n+1}\to \f(S)_n \to \dotsc \to \Id_{\SH(S)} \]
whose associated graded $\s(S)_n$ is the $n$-th effective slice functor.
Here ``$(S)$'' is often removed to simplify the notation.
A motivic spectrum in $\SH(S)_{\geq \{n\}\times \{n\}}$ is
called $n$-\textit{very effective} (or simply \textit{very effective} if $n=0$).
The composition
$\inc(S)_{\{n\}\times \{n\}}\circ \ret(S)_{\{n\}\times \{n\}} =: \vf(S)_n$
(usually denoted $\tilde{\f}_n$)
is called the $n$-very effective cover;
these give rise to the {\O}stv{\ae}r-Spitzweck filtration
\[ \dotsc \to \vf(S)_{n+1}\to \vf(S)_n \to \dotsc \to \Id_{\SH(S)} \]
whose associated graded $\vs(S)_n$ (usually denoted $\tilde{\s}_n$)
is the $n$-th very effective slice functor.
In both cases ``$(S)$'' is often removed to simplify the notation.
A motivic spectrum in $\SH(S)_{\geq \{n\}\times \Z}$ is
called $n$-\textit{connective} (or simply \textit{connective} if $n=0$).
The corresponding cover is usually denoted by a subscript ``${}_{\geq n}$''.
If $S$ is the spectrum of a field, a motivic spectrum $\E$ is $n$-connective
if and only if the homotopy sheaves $\pis_{s+(\star)}\E=0$ vanish for $s<n$
by work of Morel \cite{morel.stable-conn}.
If $S$ is more generally a locally Noetherian scheme of
finite Krull dimension $d$ and $\E\in \SH(S)$ is $n$-connective,
$\pis_{s+(\star)}\E =0$ vanishes for $s<n-d$ by \cite[Theorem]{druzhinin.conn}.
Moreover, the presheaf version \cite[Theorem V]{druzhinin.conn} shows
that the homotopy groups $\pi_{s+(\star)}\E=0$ vanish for $s<n-d$.
A stronger unstable version based on valuative dimension
is given in \cite{bouis-kundu}.

In a somewhat orthogonal direction
given in \cite{dugger-isaksen.cell}, let $\SH^\cell(S)$ be the full
subcategory generated
under colimits by all spheres $\{S^{s+(w)}\in \SH(S),s,w\in \Z\}$.
A motivic spectrum in $\SH^\cell(S)$ is called \textit{cellular}.
The inclusion of the cellular motivic spectra admits a right adjoint
compatible with colimits, and the resulting counit will be
denoted $\E^\cell \to \E$.
It induces an isomorphism on homotopy groups, but an isomorphism
of all homotopy sheaves if and only if $\E$ is cellular.
Examples of cellular motivic spectra are the periodic $K$-theory
spectra $\KGL$ (by loc.~cit.) and $\KQ$ \cite{kumar-roendigs, rso.cell}.
Note that $(-)^\cell$ preserves connectivity, but not necessarily effectivity.

\section{Hermitian $K$-theory}
\label{sec:hermitian-k-theory}

Let $\KQ_S\in \SH(S)$ denote the motivic ring spectrum
constructed in \cite{chn}, and let $\kq_S:=\vf_0\KQ_S\to \KQ_S$
denote its very effective cover, a motivic ring spectrum in its own right
by \cite{grso}.
Similarly, let $\KGL_S\in \SH(S)$ denote the motivic ring spectrum
representing homotopy $K$-theory, with (very) effective cover
$\kgl_S:=\vf_0\KGL_S$.
The forgetful and hyperbolic functors fit into a cofiber sequence
\begin{equation}\label{eq:woodKQ}
  \Sigma^{(1)}\KQ_S\xrightarrow{\eta\smash \KQ_S} 
  \KQ_S \xrightarrow{\Forg} \KGL_S
  \xrightarrow{\Sigma^{1+(1)}\Hyper\circ (\beta\smash\KGL_S)^{-1}}
  \Sigma^{1+(1)} \KQ_S
\end{equation}
in which the forgetful functor is a ring map, and $\beta\in \pi_{1+(1)}\KGL_S$
is ``the'' Bott periodicity element \cite{chn}.
The very effective version of this cofiber sequence
from \cite{aro} and \cite[Theorem 4.3]{kro} will be used, stated
here for convenience.

\begin{theorem}
\label{thm:woodkq}
If $S$ is an essentially smooth scheme over a Dedekind scheme $\mathcal{D}$, 
the very effective covers of algebraic and hermitian $K$-theory fit into the cofiber sequence
\[ 
\Sigma^{(1)}\kq_S \xrightarrow{\eta\smash \kq_S} 
\kq_S \xrightarrow{\vf_0\Forg} \kgl_S \to 
\Sigma^{1+(1)} \kq_S.
\]
Moreover, 
the canonical map
$\Sigma^{1+(1)}\kgl_S \simeq \vf_1\KGL_S \to \vf_0\KGL_S = \kgl_S$
composes with the last map to 
$\Sigma^{1+(1)}\vf_0\Hyper$.
\end{theorem}

Implicit in the last statement of Theorem~\ref{thm:woodkq} is the form of the
(very) effective filtration on algebraic $K$-theory.
Let $\beta\in \pi_{1+(1)}\kgl_S$ denote the lift of the Bott element
$\beta\in \pi_{1+(1)}\KGL_S$.

\begin{theorem}[Voevodsky, Levine, Bachmann-Elmanto-Morrow]\label{thm:zeroslicecof-kgl}
  Let $S$ be a qcqs scheme, and $\KGL_S\in \SH(S)$ the absolute
  motivic ring spectrum representing homotopy algebraic $K$-theory.
  The very effective and effective filtrations coincide: $\vf_q\KGL_S\simeq \f_q\KGL_S$
  for all $q\in \Z$.
  The cofiber sequence
  \[ \f_1\KGL_S\to \f_0\KGL_S \to \s_0\KGL_S \to \Sigma \f_1\KGL_S\]
  is equivalent to the cofiber sequence obtained by multiplication with $\beta$ 
  \[ \Sigma^{1+(1)}\kgl_S\xrightarrow{\beta\smash \kgl_S}\kgl_S
    \to \HZ_S \to \Sigma^{2+(1)}\kgl_S\]
  in which $\kgl_S\to \HZ_S$ is the unique ring map to
  the integral Eilenberg-MacLane spectrum.
  The composition $\HZ_S\to \Sigma^{2+(1)}\kgl_S\to \Sigma^{2+(1)}\HZ_S$
  coincides with the unique nonzero map
  $\HZ\to \HZ/2\xrightarrow{\Sq^2}\Sigma^{1+(1)}\HZ/2
  \xrightarrow{\partial^2_\infty} \Sigma^{2+(1)} \HZ$.
\end{theorem}

\begin{proof}
  The absolute
  motivic spectrum $\KGL_S$ was constructed by Voevodsky \cite{voevodsky.icm}.
  Representability of algebraic $K$-theory for regular schemes follows from \cite{mv},
  the general case is \cite{cisinski}. Its $E_\infty$-ring structure may be
  deduced from \cite{gepner-snaith} or \cite{rso.bott}.
  The statement on the (very) effective slice filtrations and the identification
  of the zero slice of $\KGL_S$ is due to Voevodsky over fields of characteristic zero
  \cite{voevodsky.zero} and due to Levine over general perfect fields
  \cite{levine.zero}; perfectness can be removed by a straightforward argument.
  The extension to qcqs schemes and the relation to Bott multiplication
  is stated in \cite[Section 4.1.2]{bem}.
  The final sentence follows from the corresponding statement over $\Spec(\Z)$,
  which is given in the proof of \cite[Lemma 8.4]{kro}.
\end{proof}

Let $\unit_S\to \kq_S$ denote the unit map, inducing
the graded ring homomorphism $\pi_0\unit_S\to \pi_0\kq_S$.
The ring maps $\kq_S\to\kgl_S\to \HZ_S$ define similar unit maps and homomorphisms.

\begin{lemma}\label{lem:pi0HZ-integers}
  The canonical map $\KMW_\star(\Z)\to \pi_{0-(\star)}\HZ_\Z$ induces
  an isomorphism
  $\KMil_\star(\Z)\iso \pi_{0-(\star)}\HZ_\Z$.
  Moreover $\pi_{n}\HZ_\Z=0$ for $n<0$.
  The same holds if $\Z$ is replaced with $\Z_{(p)}$ for $p>5$.
  If $p\leq 5$, the homomorphism
  $\KMil_\star(\Z_{(p)})\to \pi_{0-(\star)}\HZ_{\Z_{(p)}}$ is an isomorphism
  for $\star < 2$ and surjective for $\star\geq 2$.
\end{lemma}

\begin{proof}
  Since $\eta\in \pi_{0-(1)}\unit_\Z$ maps to $0\in \pi_{0-(1)}\HZ_\Z=0$
  (hence over
  any base scheme $S$), the canonical map
  $\KMW_\star(\Z)\to \pi_{0-(\star)}\HZ_\Z$ induces a ring homomorphism
  $\KMil_\star(\Z)\to \pi_{0-(\star)}\HZ_\Z$.
  Its source has been described as a graded ring in
  \ref{ex:kmil-kwitt-integers},
  its target is given as a graded group for example in \cite[Theorem 10.1]{kro},
  which also supplies the last statement.
  Together with \cite[Lemma 10.2]{kro}, it follows that the target graded ring
  is abstractly isomorphic to the source graded ring.
  It remains to prove that the given ring homomorphism is a concrete isomorphism.
  This follows in degree $\star=0$, just because it is a ring homomorphism.
  It remains to observe that the homomorphism is nontrivial in degrees $\star >0$,
  which can be seen by base change from $\Z$ to the field of rational or real numbers.

  The statements on $\Z_{(p)}$ may be deduced as follows.
  The vanishing $\pi_{s}\HZ_R=0$ for $s<0$
  holds more generally for any subring $R\subset \Q$, just
  as the isomorphism statement for $\star \neq 2$.
  Already
  the image of $\KMil_\star(\Z_{(p)})\to\KMil_\star(\Q)$ is described
  in \cite[Lemma 2.1]{milnor.inv}, which provides
  surjectivity. Matsumoto's Theorem and the
  injectivity of $K_2(\Z_{(p)})\to K_2(\Q)$ given in~\cite{dennis-stein}
  imply the isomorphism
  for $p>5$ with \cite{vdkallen}.
  Injectivity in the final sentence is obvious, the
  surjectivity can be deduced from \cite{kerz.improved}.
\end{proof}

The (very) effective filtration on $\kgl$ allows to lift information 
from motivic cohomology to the
higher homotopy groups of $\kgl$. For the following, $\pi_0\kgl_R$ and
$\pi_1\kgl_R$ are sufficient.

\begin{proposition}\label{prop:pi0kgl-integers}
  Let $R\subset \Q$ be a subring.
  Then the homomorphism
  $\pi_{s-(\star)}\kgl_R\to \pi_{s-(\star)}\HZ_R$
  induced by the canonical map $\kgl_R\to \s_0\kgl_R$
  is an isomorphism for $s\leq 0$ and surjective for $s=1$.
  The kernel of $\pi_{1-(\star)}\kgl_R\to \pi_{1-(\star)}\HZ_R$
  is isomorphic to the cokernel of
  $\pi_{2-(\star)}\HZ_R\to \pi_{0-(\star)}\Sigma^{(1)}\HZ_R$ induced by the
  composition
  \[ \HZ\to \HZ/2\xrightarrow{\Sq^2}\Sigma^{1+(1)}\HZ/2
    \xrightarrow{\partial^2_\infty} \Sigma^{2+(1)} \HZ.\]
\end{proposition}

\begin{proof}
  Consider the long exact sequence
  \[ \dotsc \to \pi_0 \Sigma^{1+(1)}\kgl_R\to \pi_0\kgl_R \to \pi_0\HZ_R\to \pi_{-1}\Sigma^{1+(1)}\kgl_Z \to \pi_{-1}\kgl_R\to 0 \]
  induced by the cofiber sequence from Theorem~\ref{thm:zeroslicecof-kgl}.
  The last entry $0=\pi_{-1}\HZ_R$, as observed
  in Lemma~\ref{lem:pi0HZ-integers},
  can be read off for example from \cite[Theorem 10.1]{kro}.
  Since a very effective motivic spectrum $\E$ over a locally Noetherian
  base of Krull dimension $d$ satisfies $\pi_n\E=0$ for $n<-d$
  by \cite{druzhinin.conn}, then $\pi_n\kgl_R=0$ for all $n<-1$.
  In particular,
  $\pi_{-1}\Sigma^{1+(1)}\kgl_R=\pi_{-2}\Sigma^{(1)}\kgl_R = 0$, giving
  $\pi_{-1}\kgl_R=0$.
  This implies $\pi_0\Sigma^{1+(1)}\kgl_R=0$, whence the first statement.
  The group
  $\pi_1\Sigma^{1+(1)}\kgl_R\iso \pi_0\Sigma^{(1)}\kgl_R\iso
  \pi_0\Sigma^{(1)}\HZ_R$, and 
  the form of the composition
  $\HZ_R\to \Sigma^{2+(1)}\kgl_R\to\Sigma^{2+(1)}\HZ_R$ has been 
  specified in Theorem~\ref{thm:zeroslicecof-kgl}.
\end{proof}

The passage from algebraic $K$-theory to hermitian $K$-theory
involves the Wood cofiber sequence. To
analyze it in detail, the slice
spectral sequence is helpful.
The effective slices of $\kq_S$ for $S$ essentially smooth over
a Dedekind scheme have been determined in
\cite[Corollary 4.4]{kro}.
Recall from \cite[Lemma 3.13]{rso.oneline} that the slice completion
of $\kq$ is naturally equivalent to the $\eta$-completion $\kq^{\wedge}_\eta$,
because $\kq/\eta\simeq \kgl$ is slice complete. This
involves Theorem~\ref{thm:woodkq} and Theorem~\ref{thm:zeroslicecof-kgl}.
Hence the slice spectral sequence for $\kq$ converges to the homotopy
groups/sheaves of $\kq^{\wedge}_\eta$, one corner of the $\eta$-arithmetic
square
\begin{equation}\label{eq:eta-kq} \begin{tikzcd}
    \kq \ar[r] \ar[d] & \kw:=\kq[\eta^{-1}] \ar[d] \\
    \kq^\wedge_\eta \ar[r]  & \kq^\wedge_\eta[\eta^{-1}] 
  \end{tikzcd}
\end{equation}
for $\kq$.

\begin{theorem}\label{thm:pi0kq-integers}
  The canonical map $\KMW_\star(\Z)\to \pi_{0-(\star)}\kq_\Z$ is an isomorphism
  of graded rings.
  Moreover $\pi_{n}\kq_\Z=0$ for $n<0$.
  The same holds if $\Z$ is replaced with $\Z_{(p)}$ for $p>5$.
\end{theorem}

\begin{proof}
  Let $R\subset \Q$ be a subring as indicated in the statement.
  The vanishing $\pi_{n}\kq_R=0$ for $n<-1$ follows as above
  from \cite[Theorem V]{druzhinin.conn}.
  Consider the end of the long exact sequence
  \[ \dotsc \to \pi_0 \Sigma^{(1)}\kq_R\to \pi_0\kq_R \to \pi_0\kgl_R\to \pi_{-1}\Sigma^{(1)}\kq_R \to \pi_{-1}\kq_R\to 0 \]
  on homotopy groups induced by the Wood cofiber sequence \ref{thm:woodkq}.
  The last entry $0=\pi_{-1}\kgl_R$ uses the
  connectivity of $\kgl_R$  already established in
  Proposition~\ref{prop:pi0kgl-integers}.
  The latter implies with Lemma~\ref{lem:pi0HZ-integers}
  that the composition $\KMW(R)\to\pi_0\kq_R\to\pi_0\kgl_R$
  is surjective, whence multiplication with $\eta$ is an isomorphism
  $\pi_{-1}\Sigma^{(1)}\kq_R\to\pi_{-1}\kq_R$.
  To prove that $\pi_{-1}\kq_R=0$, it thus suffices to prove that
  $\pi_{-1-(w)}\kq_R=0$
  for one weight $w$. The weight $w=-1$ works.
  In fact, the determination of the very effective slices of $\KQ_R$
  in \cite[Theorem 4.1]{kro} (extending Bachmann's
  computation \cite{bachmann.hermitian} for fields
  of characteristic not two) shows in particular $\kq_R\simeq \vf_{-1}\KQ_R$,
  because $\vs_{-1}\KQ_R\simeq \ast$.
  Hence using very effectivity,
  $\pi_{-1-(1)}\kq_R\iso \pi_{-1-(1)}\vf_{-1}\KQ_R\iso \pi_{-1-(1)}\KQ_R$.
  Since the first symplectic $K$-group $\pi_{-1-(2)}\KQ_R$
  of $R$ vanishes and the
  hyperbolic map injects $K_0(R)=\pi_{-1-(1)}\KGL_R\iso \Z$ into the
  zeroth symplectic $K$-group $\pi_{-2-(2)}\KGL_R$, the group
  $\pi_{-1-(1)}\KQ_R=0$. Here the representability of symplectic
  $K$-theory by $\KQ$ in weights congruent to $2$ modulo $4$ is used,
  deducible for example from \cite[Remark R.7]{9.III}.
  Thus $\pi_{-1}\kq_R=0$. 

  In order to specify the homomorphism $\pi_1\kq_R\to\pi_1\kgl_R$, 
  the effective slice spectral sequence will be used.
  The slice filtration on $\pi_{0+(\star)}\kq_R$
  may be analyzed as follows, based on the computation of the slices
  in \cite[Corollary 4.4]{kro}.
  Using very effectivity and representability of
  the Grothendieck-Witt group again, the canonical homomorphism
  $\KMW_0(R)=\pi_{0+(0)}\KQ_R=\pi_{0+(0)}\kq_R \to \pi_{0+(0)}\s_0\kq_R=H^{0,0}(R)=\Z$
  coincides with the rank homomorphism, as one sees by passage to
  an algebraic closure of $\Q$.
  Moreover $\pi_{1+(0)}\s_0\kq_R= H^{-1,0}(R)=0$, whence
  $\pi_{0+(0)}\f_1\kq_R$ identifies with
  the kernel of the rank homomorphism, also known as the fundamental ideal
  $I(R)$.
  The multiplicative structure of the slice
  filtration \cite{zbMATH05870074,grso} then provides inclusions
  $I^q(R) \subset \pi_{0+(0)}\f_q\kq_R$ for all $q\geq 0$. 
  Localization for the classical Grothendieck-Witt ring, generalized
  for example in \cite[Theorem 7]{9.III}, supplies an
  injection $\pi_{0+(0)}\kq_R\hookrightarrow \pi_{0+(0)}\kq_\Q$, hence
  inclusions $\pi_{0+(0)}\f_q\kq_R \subset \pi_{0+(0)}\f_q\kq_\Q = I^q(\Q)$
  for all $q\geq 0$.
  In particular, the slice filtration on $\pi_{0+(0)}\kq_R$ is separated
  by \cite{arason-pfister.krull}.

  The analogous localization sequence for motivic cohomology gives
  injections $\pi_1\HZ_R\to \pi_1\HZ_\Q$. Since
  $\pi_{1}\kq_\Q\to \pi_{1}\s_0\kq_\Q$ is surjective by the paragraph
  after \cite[Prop.~27]{aro}, the same is true for $R$ in place of $\Q$.
  The forgetful map $\kq\to \kgl$ induces an equivalence
  on zero slices and
  the connecting map $\Sigma^{(1)}\HZ/2\to \Sigma^{1+(1)}\HZ$ on first
  slices by \cite[Corollary 4.4]{kro} and the known behaviour
  over $\Q$ given in \cite{RO:slices}.
  The description of $\pi_1\kgl_R$ given in
  Proposition~\ref{prop:pi0kgl-integers}
  then shows that the image of the homomorphism
  $\pi_1\kgl_R/\pi_1\kq_R\to \pi_0\Sigma^{(1)}\kq_R$ 
  coincides with the image of a homomorphism
  $\pi_0\Sigma^{(1)}\HZ_R\to \pi_{0}\Sigma^{(1)}\kq_R$
  fitting into a commutative diagram
  \begin{equation}\label{eq:kmw-wood}
    \begin{tikzcd}
      \KMil_{w+1}(R)\ar[r,"\hyp"] \ar[d] &
      \KMW_{w+1}(R) \ar[r,"\eta"]\ar[d] & \KMW_w(R) \ar[r] \ar[d] &
      \KMil_w (R)  \ar[r] \ar[d] & 0 \\
      \pi_{0-(w)}\Sigma^{(1)}\HZ_R \ar[r] & 
      \pi_{0-(w+1)}\kq_R  \ar[r,"\eta"] &
      \pi_{0-(w)}\kq_R \ar[r] & \pi_{0-(w)}\kgl_R \ar[r] & 0
    \end{tikzcd}
  \end{equation}
  in which the outer vertical homomorphisms are surjective by
  Proposition~\ref{prop:pi0kgl-integers} and Lemma~\ref{lem:pi0HZ-integers}.
  The lower horizontal row comes from the Wood cofiber sequence
  in Theorem~\ref{thm:woodkq}
  and is thus exact.
  Suppose now that $R$ is such that the outer vertical homomorphisms
  in diagram~(\ref{eq:kmw-wood}) are isomorphisms.
  By induction, the exactness of the upper row will be established, and
  the isomorphism property will be shifted from $w$ to $w+1$.

  To start the induction, observe that by
  very effectivity $\pi_{0-(w)}\kq_R \to \pi_{0-(w)}\KQ_R$
  is an isomorphism for $w\leq 0$. The target represents
  the Witt group of symmetric bilinear forms over $R$ if $w< 0$
  and the Grothendieck-Witt group of symmetric bilinear forms over $R$
  if $w=0$.
  The source represents the same group if $R=\Z$ by
  Example~\ref{ex:kmw-integers} and if $R=\Z_{(p)}$
  for $p>3$ by \cite[Theorem 5.4]{gsz} (and for $p=3$ by
  \cite[Theorem 1.3]{schlichting-rogers}). Moreover,
  the canonical map $\KMW_w(R)\to \pi_{0-(w)}\kq_R$ is an
  isomorphism for $w\leq 0$.
  Inductively, one proves exactness of the
  upper row in the diagram~(\ref{eq:kmw-wood}) at the
  relevant spot $\KMW_{w+1}(R)$, and then shifts the
  surjectivity one to the left. To do the same for
  injectivity, observe that the homomorphism
  $\pi_{0-(w)}\Sigma^{(1)}\HZ_R\to \pi_{0-(w)}\Sigma^{(1)}\kq_R$
  is injective for $w\leq 2$, since then the source of
  the incoming differential given in Proposition~\ref{prop:pi0kgl-integers}
  is zero, and the zero homomorphism for $w>2$ by
  \cite[Theorem 10.1, Lemma 10.2]{kro}. This holds actually for
  every subring $R\subset \Q$, because
  $H^{w+1,w+1}(\Z)\to H^{w+1,w+1}(R)\to H^{w+1,w+1}(\Q)$
  are isomorphisms if $w>2$.
  However, then also $\KMil_{w+1}(R)\xrightarrow{\hyp} \KMW_{w+1}(R)$
  is the zero homomorphism, again by comparison with $\Z$.
\end{proof}

Having established the desired computation for the very effective hermitian
$K$-theory spectrum over $\Z$, it remains to lift it to the
motivic sphere spectrum over $\Z$.
One ingredient will be the corresponding situation over fields $F$ of
characteristic $2$, which had to be excluded from large parts of
\cite{rso.oneline, rso.twoline} due to the lack of a good
motivic spectrum $\kq_F$ at the time of writing.

\begin{proposition}\label{prop:kmw-kq-char-2}
  Let $F$ be a field of characteristic $2$. The canonical map
  $\KMW_\star(F)\to \pi_{0-(\star)}\kq_F$ is an isomorphism.
\end{proposition}

\begin{proof}
  As in the proof of Theorem~\ref{thm:pi0kq-integers}, effectivity implies
  that $\pi_{0-(\star)}\kq_F\iso \pi_{0-(\star)}\KQ_F$ for $\star\leq 0$. The
  latter groups are the Grothendieck-Witt group of $F$ for $\star=0$ and
  the Witt group of $F$ for $\star<0$, based on
  non-degenerate symmetric bilinear
  forms. In fact, this coincidence between ``homotopy symmetric'' and
  ``classical symmetric'' holds for any Dedekind ring in place of $F$
  \cite[Theorem 6]{9.III}.
  Morel showed in \cite[Lemma 3.10]{morel.at} that $\KMW_\star(F)$ presents
  these groups for $\star \leq 0$,
  thereby giving the statement for $\star\leq 0$.
  One may then proceed by induction as in the proof of
  Theorem~\ref{thm:pi0kq-integers}. Here
  the map $\pi_{2}\HZ_F\to \pi_0\Sigma^{(1)}\HZ_F$
  occurring in Proposition~\ref{prop:pi0kgl-integers} is
  always zero by \cite{geisser-levine.p}, giving that the latter injects into
  $\pi_{0}\Sigma^{(1)}\kq_F$.
  Also $\KMil_{w+1}(F)\xrightarrow{\hyp}\KMW_{w+1}(F)$ is injective,
  because $\hyp=2+\eta[-1]=2+\eta[1]=2$ and $\KMil(F)$ has
  no $2$-torsion by loc.~cit. This allows to conclude the induction.
\end{proof}

For the sake of completeness, the coefficients of $\kw_F=\kq_F[\eta^{-1}]$
deserve to be recorded.

\begin{proposition}\label{prop:pikwchar2}
  Let $F$ be a field of characteristic $2$. There exists a
  canonical element $\alpha\in \pi_{2+(2)}\kw_F$ such that
  \[ \pis_{\ast+(\star)}\kw_F\iso \Witt[\eta^{\pm 1},\alpha] \]
  as graded rings, where $\Witt$ denotes the unramified Witt sheaf
  on $\Sm_F$. The square of the element $\alpha$ coincides with
  the image of the Bott periodicity element $\beta^4\in \pi_{4+(4)}\kq_\Z$
  in $\pi_{4+(4)}\kw_F$. 
\end{proposition}

\begin{proof}
  As documented below \cite[Remark 8.1.2]{chn}, $\KQ$, and
  hence $\KW$, admits Bott periodicity with respect to a
  distinguished element in $\pi_{4+(4)}\KQ_\Z$ which
  lifts the fourth power of the Bott periodicity element
  $\beta\in \pi_{1+(1)}\KGL_\Z$.
  As a suggestive abuse of notation, let $\beta^4$ denote the
  lift to $\pi_{4+(4)}\KQ_\Z$, even though it is not a fourth power
  of something in $\pi_{1+(1)}\KQ_\Z$.
  Let $\beta^4$ also denote its image in $\pi_{4+(4)}\KW_\Z$, as
  well as lifts to the very effective cover $\pi_{4+(4)}\kq_\Z$ and
  the connective cover $\kw:=\kq[\eta^{-1}]$.
  To justify this terminology, observe first that
  \cite[Remark 8.1.8]{chn}
  implies a canonical equivalence $\KW\simeq \KQ[\eta^{-1}]$.
  Then $\kw$ coincides with the connective cover of 
  $\KW$, because for every $s\geq 0$ the canonical
  map $\kw\to \KW$ induces an isomorphism
  \[ \pis_s\kw \iso \colim_{m\to \infty} \pis_{s+(m)}\kq\iso
    \colim_{m\to \infty} \pis_{s+(m)}\KQ\iso \pis_s\KW \]
  of sheaves of homotopy groups.
  
  Now suppose that $F$ is a field of characteristic $2$. 
  The $2$-periodicity stated
  in~\cite[Cor.~R.10]{9.III} provides an element
  $\alpha$ in the correct degree, namely the unit in
  $\pi_{2+(2)}\KW_F\iso \Witt(F)$. For the sake of uniqueness,
  it suffices to observe by absoluteness of $\KW$
  that for $F=\FF_2$, the unit is the
  only nontrivial element. To conclude that its square is
  $\beta^4$, the multiplicative structure on the slices of
  $\KW$ recorded in \cite{ro.mult} now transfers to any scheme
  essentially smooth over a Dedekind scheme, thanks to
  \cite[Corollary 4.4]{kro}.
  Since $\alpha$ corresponds to the unique nonzero element in
  $\pi_{2+(2)}\s_2\KW_{\FF_2}\iso h^{0,0}(\FF_2)$, its square is
  sent to the unique nonzero element in $\pi_{4+(4)}\s_4\KW_{\FF_2}$,
  which corresponds to $\beta^4$. Here the vanishing
  of $\pi_s\HZ_F/2$ for $s>0$ for fields of
  characteristic $2$ enters \cite{geisser-levine.p}.
  It follows that
  $\pis_{\ast+(\star)}\KW_F\iso \Witt[\eta^{\pm 1},\alpha^{\pm 1}]$, where
  $\Witt$ denotes the unramified Witt sheaf on $\Sm_F$ obtained from
  symmetric bilinear forms.
  Hence $\pis_{\ast+(\star)}\kw_F$ has the claimed form.
\end{proof}

The vanishing of $\pi_s\HZ_F/2$ for $s>0$ from \cite{geisser-levine.p}
used in the proof of Proposition~\ref{prop:pikwchar2}
implies that on the start page of the slice spectral sequence for 
$\pi_{s+(w)}\kq_F$, the columns for $s$ odd and positive consist of
of at most $2s-1$ nonzero entries, namely the integral motivic cohomology
groups $H^{s-w-2,s-w-1}(F),H^{s-w-4,s-w-3}(F),\dotsc, H^{-w-s,-w}(F)$.
In particular, the homotopy groups $\pi_{s+(\star)}(\kq_F)^{\wedge}_\eta[\eta^{-1}]$ are
concentrated in even nonzero degrees, just as $\pi_{s+(\star)}\kw$.
In fact, more is true.

\begin{lemma}\label{lem:pilastcorner}
  Let $F$ be a field of characteristic $2$.
  The periodicity element $\alpha$ gives rise to
  an isomorphism
  \[ \pi_{\ast+(\star)}(\kq_F)^{\wedge}_\eta[\eta^{-1}] \iso
    \Witt(F)^{\wedge}_I[\eta^{\pm 1},\alpha]\]
  of graded rings.
\end{lemma}

\begin{proof}
  The canonical map $\kw\to \kq^\wedge_\eta[\eta^{-1}]$ is a map of motivic
  ring spectra. The description of $\alpha$ from the
  proof of Proposition~\ref{prop:pikwchar2} and the multiplicative structure
  of the slices of $\kq$ given in \cite[Corollary 4.4]{kro} imply that 
  multiplication with $\alpha$ induces an isomorphism from the column
  indexed by the even number $s$ to the column indexed by $s+2$
  on the first page of
  the slice spectral sequence, provided the weight $w$ is greater than $s$.
  Since in the determination of $\pi_{s+(w)}(\kq_F)^{\wedge}_\eta[\eta^{-1}]$
  one may increase the weight along the inversion of $\eta$, the
  polynomial structure follows. It remains to see that
  $\pi_{0+(0)}(\kq_F)^\wedge_\eta[\eta^{-1}]\iso \Witt(F)^\wedge_I$ is the
  completion of the Witt ring of $F$ at the fundamental ideal.
  This follows from the proof of Proposition~\ref{prop:kmw-kq-char-2}.
\end{proof}

\begin{proposition}\label{prop:pi1kqchar2}
  Let $F$ be a field of characteristic $2$.
  The canonical map $\kq_F\to \s_0\kq_F\simeq \HZ_F$ induces
  an isomorphism
  $\pi_{1+(w)}\kq_F\iso \pi_{1+(w)}\HZ_F = H^{-w-1,-w}(F)$.
  In particular, this group is zero for $w>-2$.
\end{proposition}

\begin{proof}
  The observation regarding $\pi_{s+(\star)}(\kq_F)^\wedge_\eta[\eta^{-1}]$
  stated above Proposition~\ref{prop:kmw-kq-char-2} supplies an exact sequence
  \[ 0 \to \pi_2 \kq_F \to \pi_2(\kq_F)^\wedge_\eta\oplus \pi_2\kw_F
    \to \pi_2(\kq_F)^\wedge_\eta[\eta^{-1}] \to \pi_1\kq_F \to
    \pi_1(\kq_F)^\wedge_\eta \to 0 \]
  induced by the $\eta$-arithmetic square~(\ref{eq:eta-kq}).
  A comparison of slice spectral sequences along the map $\kq_F\to \kw_F$
  implies that the first two columns consist of permanent cycles.
  In particular, $\pi_1(\kq_F)^\wedge_\eta\iso \pi_1\s_0\kq_F$.
  It remains to see that the homomorphism
  $\pi_2(\kq_F)^\wedge_\eta\oplus \pi_2\kw_F
  \to \pi_2(\kq_F)^\wedge_\eta[\eta^{-1}]$ is surjective.
  
  Lemma~\ref{lem:pilastcorner} implies that already
  $\Witt(F)\iso\pi_{2+(\star)}\kw_F\to\pi_{2+(\star)}(\kq_F)^\wedge_\eta[\eta^{-1}]
  \iso \Witt(F)^\wedge_I$ is an isomorphism if $\Witt(F)$ is complete
  with respect to the fundamental ideal.
  The behaviour of $\eta$ on slices of $\kq_F$ (which coincides with
  the one for all other fields given in \cite[Section 4.3]{RO:slices} by
  comparison with $\Z$ and base change) shows that already
  $\pi_{2+(w)}(\kq_F)^\wedge_\eta\to \pi_{2+(w)}(\kq_F)^\wedge_\eta[\eta^{-1}]$
  is surjective for $w\geq 2$.
  
  Suppose that $w<2$, and let
  $x\in \pi_{2+(w)}(\kq_F)^\wedge_\eta[\eta^{-1}]\iso \Witt(F)^\wedge_I$.
  Writing $x=(x_n)_{n\in \N}$ with $x_n\in \Witt(F)/I^n$, let
  $x^\prime\in \Witt(F)$ be any preimage of $x_{2-w+1}$, and
  let $\phi\colon \Witt(F)\to \Witt(F)^\wedge_I$ denote the canonical
  homomorphism.
  Then $x+\phi(x^\prime)$ is contained in the image of $\KMW_{2-w}(F)^\wedge_I$,
  which in turn is contained in the
  image of $\pi_{2+(w)}(\kq_F)^\wedge_\eta\to
  \pi_{2+(w)}(\kq_F)^\wedge_\eta[\eta^{-1}]$.
  This proves the claimed surjectivity.
\end{proof}

To conclude this section, further properties for $\kq_S$ to be used later
will be discussed. The first is its invariance under base change,
also called absoluteness.
The motivic spectrum $\KQ_S$ is absolute in the sense that, for every
morphism $f\colon X\to Y$ of qcqs schemes, there is a canonical
equivalence $f^\ast \KQ_Y\simeq \KQ_X$ \cite[Prop.~8.2.1]{chn}.
This property does not extend automatically to its very effective
cover $\kq_S$, but a limited absoluteness is achievable.

\begin{proposition}\label{prop:kq-absolute}
  Let $g\colon R\to S$ be a morphism of schemes, each of which is essentially
  smooth over a Dedekind scheme.
  Then the canonical map $g^\ast\kq_S\to \kq_R$ is an equivalence.
\end{proposition}

\begin{proof}
  Note first that the canonical map $g^\ast \KQ_S\to \KQ_R$ is an equivalence
  by \cite[Prop.~8.2.1]{chn}, in fact for any morphism of qcqs schemes.
  (The proposition is probably valid in that case.)
  Since $g^\ast$ preserves very effectivity, there results a canonical
  map $\phi\colon g^\ast \kq_S\to \kq_R$.
  To check that it is an equivalence, it suffices by
  \cite[Prop.~B.3]{bachmann-hoyois.norms} to prove that
  $i^\ast(\phi)$ is an equivalence for every point $i\colon \Spec(F)\to S$.
  This reduces the task to the situation where $g\colon \Spec(F)\to S$
  is the inclusion of a point in a Dedekind scheme $S$.
  By \cite[Theorem 4.5]{bachmann.eta-dedekind}, this holds if
  $F$ is of characteristic not $2$, passing from $S$ to an open subscheme
  containing $g(\Spec(F))$ and no point of residue characteristic $2$.
  Note that pullback along (essentially) smooth morphisms preserves
  very effective covers.
  This also reduces the task to the case where $g\colon \Spec(F)\to S$
  is the inclusion of a closed point of characteristic $2$ in the
  spectrum $S$ of a discrete valuation ring.
  
  Already at an earlier stage, but also now,
  quoting \cite[Lemma 2.16]{bachmann-hopkins}
  reduces the problem to establishing that both
  $\phi[\eta^{-1}]$ and $\phi/\eta$ are equivalences. In the latter
  case, absoluteness for $\kgl\simeq \kq/\eta$
  has been documented in \cite{bem}.
  It remains to conclude that the canonical map
  $\phi[\eta^{-1}]\colon g^\ast \kw_S\to \kw_F$ is an equivalence.
  The $\eta$-inverted sphere $\unit[\eta^{-1}]$ is absolute,
  and the unit map $\unit_S[\eta^{-1}]\to \kw_S$ has a $4$-connective
  cofiber by Corollary~\ref{cor:sphere-etainv-kw-conn}.
  It follows that $\pis_sg^\ast \kw_S\to \pis_s\kw_F$ is an isomorphism
  for all $s<3$ and surjective for $s=3$.
  Multiplication with the periodicity element
  $\beta^4\in \pi_{4+(4)}\kw_S$ induces
  an equivalence $\Sigma^{4+(4)}\kw_S\xrightarrow{\simeq} (\kw_S)_{\geq 4}$.
  Hence multiplication with
  $g^\ast \beta^4\in \pi_{4+(4)}g^\ast\kw_S$ induces an equivalence
  $\Sigma^{4+(4)}g^\ast \kw_S\xrightarrow{\simeq} g^\ast((\kw_S)_{\geq 4})$.
  Since this motivic spectrum is $4$-connective, it maps
  canonically to $(g^\ast(\kw_S))_{\geq 4}$.
  To prove that this map is an equivalence, it suffices to prove
  that $g^{\ast}((\kw_S)_{<4})$ admits only zero maps from
  any $4$-connective motivic spectrum $\E^\prime$, which
  may be taken to be of the form $g^\ast \E$ for some
  $4$-connective motivic spectrum $\E$ over $S$.

  Let $j\colon U\hookrightarrow S$ be the open embedding of the
  complement of the closed point. So $U$ is the spectrum
  of the fraction field of $S$. The localization
  cofiber sequence
  \[ j_\sharp j^\ast \to \Id \to g_\ast g^\ast \]
  implies that $[\E,g_\ast g^{\ast}((\kw_S)_{<4})]$ embeds
  in $[\Sigma^{-1}\E,j_\sharp j^{\ast}((\kw_S)_{<4})]$.
  To prove that the latter is the zero group, observe that
  the canonical map $(\kw_S)_{<4}\to (\kw_S)_{<3}$ induces an
  equivalence $j^\ast((\kw_S)_{<4}\to (\kw_S)_{<3})$.
  This uses that $j^\ast$, being induced by a smooth morphism,
  commutes with truncations, preserves $\kw$, and
  $\pi_3\kw_U\iso 0$.\footnote{If $U$ has characteristic not $2$, one may replace $3$ with $1$.}
  Since $(\unit_S[\eta^{-1}])_{<3}\simeq (\kw_S)_{<3}$ by
  Corollary~\ref{cor:sphere-etainv-kw-conn}, absoluteness for
  $\unit[\eta^{-1}]$ implies that 
  $[\Sigma^{-1}\E,j_\sharp j^{\ast}((\unit_S[\eta^{-1}])_{<3})]=0$.
  
  The equivalence $\Sigma^{4+(4)}g^\ast (\kw_S)\simeq (g^\ast(\kw_S))_{\geq 4}$
  can be iterated. Together with the already known connectivity of
  the map $g^\ast\kw_S\to \kw_F$, this implies that it is an isomorphism
  on all homotopy sheaves $\pis_s$, except perhaps for positive indices
  $s\equiv 3(4)$.
  To conclude the isomorphism for these remaining cases, observe that
  by Proposition~\ref{prop:pikwchar2}
  both the source and the target of this map
  acquire a compatible periodicity element
  in $\pi_{2+(2)}g^\ast\kw_S\iso \pi_{2+(2)}\kw_F$ which is a square root
  of the one already used. However, these
  periodicity elements only come from
  $S$ if the fraction field of $S$ has characteristic $2$;
  the argument above may be suitably modified in that case.
  To conclude for the mixed characteristic case, one may cite
  \cite[Remark 2.3.5]{9.III} to get that
  in the localization cofiber sequence already employed,
  both $\pi_2\kw_S$ and $\pi_3\kw_S$ are $0$.
  This implies an isomorphism $\pi_{3}g^\ast\kw_S\iso \pi_2j_\sharp j^\ast\kw_S$.
  Since one may pass to the $3$-connective cover
  of $\kw_S$ in order to determine $\pi_3$, and
  the map $j^\ast (\kw_{\geq 4} \to \kw_{\geq 3})$ is an equivalence,
  the motivic spectrum $j_\sharp j^\ast \kw_{\geq 3}$ is $4$-connective
  over a Noetherian base of Krull dimension $1$. Then
  $\pi_2j_\sharp j^\ast \kw_{\geq 3}=0$ by \cite[Theorem V]{druzhinin.conn}.
  As this argument for vanishing homotopy groups works over any field
  of characteristic $2$, it holds on the level of homotopy sheaves
  by \cite{morel.at}. The equivalence $g^\ast \kw_S\to \kw_F$ follows.
\end{proof}

The absolute motivic spectrum $\KQ$ is known to satisfy absolute
purity by \cite[Theorem 8.4.2]{chn}, 
which extends \cite[Theorem D.1]{dfjk} over schemes in which $2$ is a unit.
This means that
for every smoothable local complete
intersection morphism $f\colon X\to Y$ of regular separated finite-type
$S$-schemes, the purity transformation
$\Sigma^{\Tan f}f^\ast\KQ_Y \to f^!\KQ_Y$
defined in \cite{djk} is an equivalence.
Here $S$ is Noetherian of finite Krull dimension,
and $\Sigma^{\Tan f}$ is suspension with the
Thom spectrum of the relative tangent bundle of $f$.
Again the extension to its very effective cover $\kq$ is not automatic,
and the base change invariance given in Proposition~\ref{prop:kq-absolute}
is too limited anyhow. The following, in which the
suspension with the trivial relative tangent bundle has been specified,
has to suffice.

\begin{proposition}\label{prop:absolute-purity-kq}
  Let $S$ be a connected Dedekind scheme whose fraction field has
  characteristic $0$.
  Let $i\colon Z\hookrightarrow S$ be the closed embedding of a point of
  positive characteristic $p$. Then the purity transformation
  $\Sigma^{-1-(1)}i^\ast\kq_S\to i^!\kq_S$ is an equivalence.
\end{proposition}

\begin{proof}
  To avoid minus signs, the suspension will be moved to the other side.
  The proof involves the Wood cofiber sequence from Theorem~\ref{thm:woodkq},
  but first the equivalence
  $i^\ast\kgl_S \to  \Sigma^{1+(1)}i^!\kgl_S$
  for the very effective algebraic $K$-theory spectrum will be obtained.
  It follows from the Bott periodicity cofiber sequence
  \[ \Sigma^{1+(1)}\kgl_S\xrightarrow{\beta}\kgl_S \to \HZ_S \to
    \Sigma^{2+(1)}\kgl_S\]
  and the corresponding purity result recorded for
  the integral motivic Eilenberg-MacLane spectrum $\HZ_S$
  in \cite[Prop.~8.3]{frankland-spitzweck}.
  Here are some details: The reduction to the case treated in
  loc.~cit.~follows by passage to the
  local ring in $S$ at the given closed point.
  Hence we may assume that $S$ is the spectrum of a discrete valuation ring
  of mixed characteristic. 
  The Bott periodicity cofiber sequence
  and the contractibility $\holim_{m\to \infty}\Sigma^{1+(1)}\kgl_S\simeq
  \holim_{m\to\infty}\vf_m\KGL_S\simeq \ast$ imply a natural equivalence
  $\kgl_S\simeq (\kgl_S)^{\wedge}_\beta:=\holim_{m\to \infty}\kgl_S/\beta^m$.
  Induction on $m$, starting with the case $\kgl_S/\beta\simeq \HZ_S$,
  provides that the purity transformation
  \[ \kgl_Z/\beta^m\simeq i^\ast\kgl_S/\beta^m \to \Sigma^{1+(1)}i^!\kgl_S/\beta^m \]
  is an equivalence for all $m$, using $i^\ast\kgl_S\simeq \kgl_Z$
  by absoluteness for $\kgl$ documented in \cite{bem}.
  Since $\Sigma^{1+(1)}$ and $i^!$ commute with homotopy limits, there
  results on homotopy limits the equivalence
  $(\kgl_Z)^{\wedge}_\beta \simeq \Sigma^{1+(1)}i^!((\kgl_S)^\wedge_\beta)$,
  which a priori may not be the purity equivalence for $(\kgl_S)^\wedge_\beta$.
  The $\beta$-completeness of $\kgl$, however, identifies it up to equivalence
  with the purity equivalence for $\kgl_S$.

  Similarly, the Wood cofiber sequence and the purity equivalence
  for $\kgl$, together with Proposition~\ref{prop:kq-absolute},
  supply an equivalence
  \[ (\kq_Z)^{\wedge}_\eta \simeq \Sigma^{1+(1)}i^!((\kq_S)^\wedge_\eta).\]
  Unfortunately, $\kq_Z$ is not necessarily $\eta$-complete.
  By the $\eta$-arithmetic square~(\ref{eq:eta-kq}) for $\kq$, the fiber of
  the $\eta$-completion map $\kq\to \kq^\wedge_\eta$ is naturally equivalent
  to the fiber $\mathsf{g}$
  of $\kq[\eta^{-1}]\to \kq^\wedge_\eta[\eta^{-1}]$.
  It suffices to prove that the map
  $\phi\colon \mathsf{g}_Z\to \Sigma^{1+(1)}i^!\mathsf{g}_S$
  is an equivalence.
  Both source and target live in the $\eta$-inverted motivic
  stable homotopy category of a field. It remains to check
  that the purity transformation induces an isomorphism
  on stable homotopy sheaves $\pis_s$ for $s\in \Z$.
  Now $\kq[\eta^{-1}]$ coincides with the connective cover
  of the motivic spectrum representing symmetric Witt groups,
  and $\kq^\wedge_\eta[\eta^{-1}]$ represents its completion with respect
  to the fundamental ideal, as proven in \cite{levine.filtration}
  for fields of characteristic not $2$; an extension to
  fields of characteristic $2$ is documented in the proof of
  Proposition~\ref{prop:kmw-kq-char-2}.
  Hence the source of $\pis_s\phi$ identifies with
  \[ \pis_s \mathsf{g}_Z \iso
    \begin{cases} \Witt/\Witt^\wedge_I &-1\leq  s \equiv 3 (4) \\
      0 & \mathrm{else} 
    \end{cases}
  \]
  if $p$ is odd, and with
  \[ \pis_s \mathsf{g}_Z \iso
    \begin{cases} \Witt/\Witt^\wedge_I &-1\leq  s\  \mathrm{odd} \\
      0 & \mathrm{else}  
    \end{cases}
  \]
  if $p=2$.
  Here $\Witt$ denotes the sheaf of Witt groups on $\Sm_Z$,
  and $\Witt^\wedge_I$ is
  its completion at the fundamental ideal.
  For the target, one uses the cofiber sequence
  \[ i_\ast i^!\mathsf{g}_S\to \mathsf{g}_S\to j_\ast j^\ast \mathsf{g}_S \to
    \Sigma i_\ast i^! \mathsf{g}_S \]
  where $j$ is the open embedding of the generic point in $S$.
  The homotopy sheaves of $j^\ast \mathsf{g}_S$ are given similarly to
  the ones for $\mathsf{g}_Z$ in the case where $p$ is odd, only with
  the Witt sheaf on $\Sm_{F(S)}$ instead, where $F(S)$ is the function field
  of $S$. In particular, these homotopy sheaves are concentrated only
  in every fourth degree.
  Let $\mathsf{G}_S$ denote the fiber of
  $\KW_S\to (\KQ_S)^\wedge_\eta[\eta^{-1}]$, so that the canonical map
  $\mathsf{g}_S\to \mathsf{G}_S$ constitutes a $-1$-connective cover.
  Using the purity transformation for $\KQ_S$, one may deduce for the
  cofiber sequence
  \[ i_\ast i^!\mathsf{G}_S\to \mathsf{G}_S\to j_\ast j^\ast \mathsf{G}_S \to
    \Sigma i_\ast i^! \mathsf{G}_S \]
  that the target $i^!\mathsf{g}_S$ has homotopy sheaves
  isomorphic to those of $\mathsf{g}_Z$ shifted by
  one. Moreover, the naturality of
  the purity transformation then implies that
  $\pis_s\phi$ induces such an isomorphism, which concludes the proof.
\end{proof}

The hermitian $K$-theory spectrum $\KQ$ is cellular over any base scheme
by \cite{kumar, kumar-roendigs}, extending \cite{rso.cell} over
schemes in which $2$ is a unit. This means that it is contained in the
full localizing subcategory generated by the spheres $S^{s+(w)}\in \SH(S)$.
Again this property does not necessarily
extend to its very effective cover $\kq$.
Before addressing the cellularity of $\kw$, the cooperations $\kw\smash \kw$
will be determined, first over fields of characteristic $2$, and then
$2$-locally over Dedekind schemes.

\begin{proposition}\label{prop:kw-coop-char2}
  Let $F$ be a field of characteristic $2$. There exists a
  decomposition $\kw_F\smash \kw_F\simeq \bigoplus_{j\geq 0}\Sigma^{4j}\kw_F$
  of right $\kw$-modules
  in which the left unit is the inclusion of the bottom summand.
\end{proposition}

\begin{proof}
  Let $\MSL$ denote the special linear algebraic cobordism spectrum,
  as introduced in \cite{panin-walter.msp}. There is a canonical
  $\SL$-orientation on $\kw$, that is, a map $\phi\colon\MSL\to \kw$ of motivic
  ring spectra \cite{ananyevskiy.sl} (which assumes the base to
  be a field of characteristic not $2$, but which now
  holds over a qcqs scheme by \cite{chn}).
  The identification $\MSL\smash \kw\simeq \kw[e_0,e_2,\dotsc]/(e_0-1)$
  from \cite[Theorem 4.1(2)]{bachmann-hopkins}, with $e_j$ having
  degree $2j+(2j)$, provides elements $\phi(e_j)$ in $\kw\smash \kw$
  for which the proof of \cite[Prop.~7.7]{bachmann-hopkins} holds;
  note that $2=2+\eta[1]=2+\eta[-1]=\hyp=0$ already on $\unit_F[\eta^{-1}]$
  if $F$ has characteristic~$2$.
\end{proof}

\begin{corollary}\label{cor:kw-coop}
  Let $S$ be essentially smooth over a Dedekind scheme.
  There exists a
  decomposition $\kw\smash \kw_{(2)}\simeq \bigoplus_{j\geq 0}\Sigma^{4j}\kw_{(2)}$
  of right $\kw$-modules
  in which the left unit is the inclusion of the bottom summand.
\end{corollary}

\begin{proof}
  It suffices to provide such a decomposition over $\Z$.
  There the identification $\MSL\smash \kw\simeq \kw[e_0,e_2,\dotsc]/(e_0-1)$
  from \cite[Theorem 4.1(2)]{bachmann-hopkins} holds; see
  \cite[Footnote 6]{bachmann-hopkins}. Here $e_j$ has
  degree $2j+(2j)$ and gives rise to elements $\phi(e_j)$ in $\kw\smash \kw$.
  That the resulting map is an equivalence after localizing at $2$
  can be checked over (prime) fields.
  It holds by \cite[Prop.~7.7]{bachmann-hopkins}
  in characteristic different from $2$ and Proposition~\ref{prop:kw-coop-char2}
  in characteristic $2$.
\end{proof}

\begin{proposition}\label{prop:kw-cellular}
  Let $S$ be essentially smooth over a Dedekind scheme. Then $\kw_S$ is
  cellular.
\end{proposition}

\begin{proof}
  By Proposition~\ref{prop:kq-absolute}, it suffices to prove the
  statement over $\Z$. Let $\gamma\colon \kw^\cell \to \kw$ denote the natural
  cellularization map.
  The goal is to show that it is an equivalence.
  Let $i\colon \Spec(\FF_2)\hookrightarrow \Spec(\Z)$ denote the closed
  embedding of the prime field with two elements, having open complement
  $j\colon \Spec(\Z[2^{-1}])\hookrightarrow \Spec(\Z)$.
  The localization cofiber sequence provides a map
  \begin{equation}\label{eq:loc-cof-cell}
    \begin{tikzcd}
      i_\ast i^! \kw^\cell \ar[r] \ar[d,"i_\ast i^!\gamma"] &
      \kw^\cell \ar[d,"\gamma"] \ar[r] &
      j_\ast j^\ast \kw^\cell \ar[d,"j_\ast j^\ast\gamma"] \\ 
      i_\ast i^! \kw \ar[r]  & \kw  \ar[r] &
      j_\ast j^\ast \kw
    \end{tikzcd}
  \end{equation}
  of cofiber sequences.
  It suffices to check that $i^!\gamma$ and $j^\ast\gamma$ are
  equivalences.
  By \cite[Prop.~5.7]{bachmann.veff-ko-kgl} and
  Proposition~\ref{prop:kq-absolute}, $j^\ast \kw$ is cellular.
  This does not imply that $j^\ast \gamma$ is an equivalence, since
  a priori it is not clear whether $j^\ast \kw^\cell$ and $(j^\ast \kw)^\cell$
  are naturally equivalent. 
  Nevertheless, the map $j^\ast\kw^\cell \to j^\ast \kw$ has
  cellular source and target.
  Hence to check that it is an equivalence, it suffices to prove that
  $\pi_{s+(\star)}j^\ast \kw^\cell \to \pi_{s+(\star)}j^\ast \kw$ is an isomorphism
  for every $s$. This will happen after a short interlude on the
  left hand side of diagram~(\ref{eq:loc-cof-cell}) involving $i^!$.
  
  Achieving the same for $i^!\gamma$ requires
  to establish that $i^\ast \kw$ is cellular.
  To do so, let $F$ be a perfect field of characteristic $2$, and let
  $\kw_F$ be defined as above. 
  Proposition~\ref{prop:pikwchar2} says
  $\pis_{\ast+(\star)}\kw_F \iso \Witt[\eta^{\pm},\alpha]$, where
  $\eta\in \pi_{(1)}\kw_F$
  is the first algebraic Hopf map, and $\alpha\in \pi_{2+(2)}\kw_F$ is
  the unique element with $\alpha^2=\beta^4$.
  (Here the image of the lift to $\kq$ of $\beta^4\in \pi_{4+(4)}\kgl$ in
  $\kw$ is still denoted $\beta^4$.)
  As above, $\Witt$ denotes the sheaf of (homotopy) symmetric Witt groups.
  Actually, both elements are characterised by being nonzero in the respective
  bidegree, since the Witt group of a perfect field of characteristic $2$
  has two elements.

  To prove cellularity of $\kw_F$, the
  cell presentation of $\kw_\CC$ from \cite[Section 4]{r.etainv}
  will be shown to work over $F$. It has the following form.
  There is a sequence of cellular motivic spectra factoring the unit of
  $\kw_F$ as
  \[
    \unit_F[\eta^{-1}] = \D_1\xrightarrow{i_1} \D_2\xrightarrow{i_2} \cdots \to \D_n\xrightarrow{i_n} \cdots \to \kw_F
  \]
  such that for every $n$ the map $\D_n\to \kw_F$ is $(4n-1)$-connective,
  with the following further properties:
  \begin{enumerate}
  \item The composition $\unit_F[\eta^{-1}]\to \D_n\to \kw_F$ induces isomorphisms on
    $\pi_{k}$ for all even $k$.
  \item For every $n\ge 1$, the map $\D_n\to \D_{n+1}$
    is the canonical map to the cone of a unique nontrivial class
    $a_n\colon \Sigma^{4n-1}\unit_F[\eta^{-1}]\to \D_n$ in
    $\pi_{4n-1}\D_n$.
  \item For every $n\geq 1$, the canonical map
    $\D_{n+1}\to \Sigma^{4n}\unit[\eta^{-1}]$
    induces an isomorphism on
    $\pis_{k}$ whenever $k\ge 4n+3$ is odd.
  \end{enumerate}
  Taking the colimit as $n\to \infty$ gives a cell presentation of $\kw$.
  Its existence follows by induction on $n$.
  The induction start is provided
  by Corollary~\ref{cor:sphere-etainv-kw-conn} below, which gives
  the correct connectivity for $n=1$.
  The first property holds for $n=1$ by
  Theorem~\ref{thm:sphere-etainv-char2} below.
  Set $a_1:=\sigma \eta^{-4}$, where $\sigma\in \pi_{3+(4)}\unit$ is the third
  algebraic Hopf map.
  Taking $\D_2$ to be its cone then fulfils the second property for $n=1$,
  and that the canonical map $\D_2\to \Sigma^4\unit[\eta^{-1}]$ induces
  isomorphisms on $\pis_k$ for $k\geq 7$ odd can again be read off from
  Theorem~\ref{thm:sphere-etainv-char2}, giving the third property for
  $n=1$.
  Suppose $\unit_F[\eta^{-1}]\to \D_n$ has been provided.
  Then $a_n\in \pi_{4n-1}\D_n\iso \pi_{4n-1}\Sigma^{4(n-1)}\unit_F[\eta^{-1}]\iso
  \Witt(F)\{\sigma\eta^{-4}\}$ is the unique nonzero element.
  Set $\D_{n+1}$ to be the cone of $a_n$.
  Then $\D_{n}\to \D_{n+1}$ is $(4n-1)$-connective by construction. 
  The homotopy sheaves $\pis_{4n-1}\D_{n+1}$ and $\pis_{4n+1}\D_{n+1}$
  are now zero, in contrast with the homotopy sheaves
  $\pis_{4n-1}\D_{n}$ and $\pis_{4n+1}\D_{n}$,
  corresponding via the canonical map to $\Sigma^{4n}\unit_F[\eta^{-1}]$ to
  $\Witt(F)$ generated by $\sigma\eta^{-4}$ and $\sigma\eta^{-6}\tilde{\alpha}$,
  respectively.
  In fact, the map $\D_n\to \D_{n+1}$ induces the zero map on all homotopy
  sheaves with odd index, and an isomorphism on all homotopy sheaves with
  even index.
  Conversely, the map $\D_{n+1}\to \Sigma^{4(n+1)}\unit_F[\eta^{-1}]$
  induces an isomorphism on all homotopy sheaves with odd index $k\geq 4n+3$,
  and the zero map on all homotopy sheaves with even index.
  The required connectivity of $\D_{n+1}\to \kw_F$ follows, as well as the
  three properties, completing the inductive construction. 

  Again it is unclear whether $i^!\kw^\cell$ and $(i^! \kw)^\cell$ are
  naturally equivalent.
  Nevertheless, $i^!\gamma$ is a map with cellular target (as just proven,
  of course using the purity equivalence from
  Proposition~\ref{prop:absolute-purity-kq})
  and source (see Remark~\ref{rem:ishriek-cellular} below).
  Thus it suffices to prove that $\pi_{s+(\star)}i^!\gamma$ is an
  isomorphism for all $s$.

  The unit $\unit[\eta^{-1}]\to \kw$ lifts uniquely via $\gamma$ to
  $\kw^\cell$, and the cofiber of this lift coincides with $\C^\cell$,
  where $\C$ is the cofiber of the unit $\unit[\eta^{-1}]\to \kw$.
  The latter is $4$-connected by Corollary~\ref{cor:sphere-etainv-kw-conn}
  below. Hence also $\C^\cell$ is $4$-connected, since $(-)^\cell$ preserves
  connectivity. The functor $j^\ast$ preserves connectivity, and
  the functor $i^!$ does so up to a shift one down for cellular motivic
  spectra in $\SH(S)[\eta^{-1}]$, again by the identification
  $i^!\simeq \Sigma^{-1-(1)}i^\ast$ on $\SH^\cell(S)[\eta^{-1}]$ implied by 
  Corollary~\ref{cor:absolute-purity-etainv}.
  Hence when applying $\pi_{s+(\star)}$ to diagram~(\ref{eq:loc-cof-cell}),
  the vertical map on the left hand side induces an isomorphism for
  $s<2$.
  Also the vertical map on the right hand side induces an isomorphism for
  $s<2$.
  While $j^\ast \C^\cell$ and $j^\ast \C$ are $4$-connective, the base
  $\Z[2^{-1}]$ having Krull dimension $1$ implies only that
  $\pi_{s+(\star)}j^\ast \C^\cell\iso 0 \iso \pi_{s+(\star)}j^\ast \C$ for $s<3$
  by \cite[Theorem V]{druzhinin.conn}.
  Nevertheless, the diagram~(\ref{eq:loc-cof-cell}) then guarantees that
  also $\pi_{2+(\star)}j^\ast \gamma$ is an isomorphism.
  Here of course information on $\pi_{s+(\star)}\kw$ is used, for example
  from \cite[Theorem 10.13]{kro} for $\Z$ and \cite{berrick-karoubi} for
  $\Z[2^{-1}]$.
  
  The cofiber sequence
  \[ j_\sharp j^\ast \to \Id_{\SH(S)} \to i_\ast i^\ast\]
  left adjoint to the cofiber sequence used already above provides the last
  homomorphism in the sequence 
  \[ \pi_{3+(\star)}j_\ast j^\ast \kw^\cell \to
    \pi_{2+(\star)} i_\ast i^! \kw^\cell \iso \pi_{3+(\star)} i_\ast i^\ast
    \kw^\cell \to \pi_{2+(\star)} j_\sharp j^\ast \kw^\cell \]
  with the purity equivalence induced by
  Proposition~\ref{prop:absolute-purity-kq}inducing the middle isomorphism.
  The isomorphisms $\pi_{3+(\star)}\kw^\cell \iso \pi_{3+(\star)}\kw\iso 0 \iso
  \pi_{2+(\star)}\kw \iso \pi_{2+(\star)}\kw^\cell$ from \cite[Theorem 10.13]{kro}
  show that the outer two
  homomorphisms are isomorphisms as well.
  The map $j^\ast\gamma$ was established as being $3$-connective,
  whence so is $j_\sharp j^\ast\gamma$.
  It follows that
  $\pi_{2+(\star)}j_\sharp j^\ast \kw^\cell \iso
  \pi_{2+(\star)}j_\sharp j^\ast \kw \iso 0$.
  Thus both $\pi_{3+(\star)}j_\ast j^\ast \kw^\cell$ and
  $\pi_{2+(\star)} i_\ast i^! \kw^\cell$ are zero, giving that
  $j^\ast \gamma$ and $i^!\gamma$ induce isomorphisms of homotopy groups
  in the lowest four degrees of interest.
    
  To increase the range, one may use the periodicity element
  $\beta^4\in \pi_{4+(4)}\kw$ (which exists over $\Z$), inducing
  the same type of periodicity in $\kw^\cell$.
  It follows that $\pi_{s+(\star)}j^\ast \gamma$ and $\pi_{s+(\star)}i^! \gamma$
  are isomorphisms for all $s$.
  Hence both $j^\ast\gamma$ and $i^!\gamma$ are equivalences, completing
  the proof that $\gamma\colon \kw^\cell \to \kw$ is an equivalence.  
\end{proof}

A consequence of the Hopkins-Morel equivalence is that $\HZ_S$, and
then also $\kgl_S$, is cellular if $S$ is essentially smooth over a field
of characteristic zero \cite[Prop.~8.1]{hoyois.mgl-hz}.
Once it is established integrally over fields of positive characteristic,
cellularity of $\kq$ then follows.

\section{The sphere spectrum}
\label{sec:sphere-spectrum}

It remains to transfer the computation of $\pi_{0+(\star)}\kq_\Z$ to
$\pi_{0+(\star)}\unit_\Z$. A key ingredient is the following result.

\begin{theorem}\label{thm:sphere-kq-eff}
  Let $S$ be an essentially smooth scheme over a Dedekind scheme $\mathcal{D}$.
  The cofiber of the unit map $\unit_S\to \kq_S$ is
  $1$-very effective and $2$-connective.
\end{theorem}

\begin{proof}
  By \cite[Lemma B.1]{bachmann-hoyois.norms}, it suffices to prove 
  the statement over the Dedekind scheme $\mathcal{D}$.
  With \cite[Prop.~B.3]{bachmann-hoyois.norms},
  it then suffices to prove it over fields.
  For fields of characteristic not two, the
  statement on very effectivity follows from
  \cite{bachmann.hermitian}. For fields of characteristic not two,
  the statement on connectivity follows from Morel's theorem on the
  zeroth Milnor-Witt stem of the motivic sphere spectrum
  (which also holds for fields of characteristic two),  
  and the surjectivity on the first Milnor-Witt
  stem \cite[Theorem 4.1]{rso.twoline}.
  
  It remains to deal with a field $F$ of characteristic two.
  Regarding connectivity: As explained in Proposition~\ref{prop:kmw-kq-char-2},
  the composition $\KMW(F)\to \pi_0\unit_F\to \pi_0\kq_F$ is an isomorphism,
  in which the first arrow depicts an isomorphism by Morel's Theorem.
  It follows that $\pi_0\unit_F\to \pi_0\kq_F$ is an isomorphism.
  In order to conclude that
  the cofiber of the unit map $\unit_F\to \kq_F$ is $2$-connective, it suffices
  to prove that $\pi_1\unit_F\to \pi_1\kq_F$ is surjective.
  Proposition~\ref{prop:pi1kqchar2} says that $\pi_1\kq_F\to \pi_1\s_0\kq_F$
  is an isomorphism.
  Since the unit $\unit_F\to \kq_F$ induces an equivalence on the zero
  slices by \cite[Corollary 4.4]{kro}, it remains to prove that
  $\pi_1\unit_F\to\pi_1\s_0\unit_F$ is surjective.
  As stated before, $\eta$ is zero on $\s_0\unit$, thus this map
  factors over $(\pi_1\unit_F)/\eta\pi_1\unit_F$.
  This $\KMW(F)$-module features in the short exact sequence
  \[ 0 \to (\pi_1\unit_F)/\eta\pi_1\unit_F\to \pi_1 \unit_F/\eta
    \to {}_\eta\pi_0\unit_F \to 0 \]
  given for example in \cite[(3.2)]{r.endo}, where the last module
  is the $\eta$-torsion in $\pi_0\unit_F\iso \KMW(F)$.
  By the exact sequence displayed in the top row of
  diagram~(\ref{eq:kmw-wood}) for $R=F$,
  it is generated as a $\KMW(F)$-module by an element
  $\hyp=2\in \pi_{0+(1)}\unit_F$.
  As in \cite[Lemma 3.4]{r.endo}, this element lifts to
  an element $\hyp^\prime \in \pi_{1+(1)}\unit_F/\eta$.
  Under the map $\unit_F/\eta\to \s_0(\unit_F/\eta)\simeq \s_0\unit_F$
  to its zero slice, any lift $\hyp^\prime\in \pi_{1+(1)}\unit_F/\eta$
  must map to $0$ in $\pi_{1+(1)}\s_0(\unit_F/\eta)=0$.
  The same then holds for any $\KMW$-multiple.
  Hence if $\pi_1\unit_F/\eta \to \pi_1\s_0(\unit_F/\eta)$ is
  surjective, then already $(\pi_1\unit_F)/\eta\pi_1\unit_F\to
  \pi_1\s_0(\unit_F/\eta)$ is surjective.
  The proof of \cite[Lemma 3.24]{rso.oneline} implies that
  $\f_1(\unit_F/\eta)$ is $1$-connective, whence
  $\pi_1\unit_F/\eta \to \pi_1\s_0(\unit_F/\eta)$ is surjective.
  This implies that $\pi_1\unit_F\to \pi_1\s_0\unit_F$ is surjective.
  Thus the cofiber of $\unit\to \kq$ is $2$-connective.

  Regarding very effectivity, \cite[Lemma 3.11(2)]{bachmann.hermitian}
  applies to show that $\vs_0(\unit_F)\to \vs_0(\kq_F)$ is an equivalence.
  This uses the already established isomorphism
  $\pis_0\unit_F\xrightarrow{\iso} \pis_0\kq_F$ and the equivalence
  $\s_0(\unit_F)\xrightarrow{\simeq} \s_0(\kq_F)$ on zero slices
  \cite[Corollary 4.4]{kro}.
  Hence the cofiber of $\unit_F\to \kq_F$ is $1$-very effective.
\end{proof}

\begin{theorem}\label{thm:pi0-sphere-z}
  Let $R\in \{\Z,\Z_{(p)}\}$.
  The unit map $\unit_R\to \kq_R$ induces isomorphisms
  $\pi_{s+(\star)}\unit_R\to \pi_{s+(\star)}\kq_R$ for all $s\leq 0$
  and a surjection for $s=1$.
  In particular, $\pi_{s+(\star)}\unit_R$ is trivial for $s<0$,
  $\pi_{0+(\star)}\unit_{\Z_{(p)}}$ is the naive Milnor-
  Witt $K$-theory of $\Z_{(p)}$ if $p>5$,
  and $\pi_{0+(\star)}\unit_\Z$ is the naive Milnor-Witt $K$-theory of $\Z$.
\end{theorem}

\begin{proof}
  Let $\C_R$ denote the cone on the unit map $\unit_R\to \kq_R$.
  There results a long exact sequence of homotopy groups ending
  in
  \[ \dotsc \to \pi_{1+(\star)}\C_R \to \pi_{0+(\star)}\unit_R
    \to \pi_{0+(\star)}\kq_R \to \pi_{0+(\star)}\C_R \to
    \pi_{-1+(\star)}\unit_R \to 0 \]
  because of Theorem~\ref{thm:pi0kq-integers} and
  \cite[Theorem V]{druzhinin.conn}.
  By Theorem~\ref{thm:sphere-kq-eff}, $\C$ is $2$-connective.
  Hence $\pi_{0+(\star)}\C_R=0$, showing that $\pi_{-1+(\star)}\unit_R=0$.
  Thus the $0$-connective motivic
  sphere spectrum over $R$ has the ``correct'' homotopy group connectivity.
  The same is then true for every
  cellular motivic spectrum over $R$.
  The unit $\unit_R\to \kq_R$ factors over the cellularization, giving
  a cofiber sequence
  \[ \unit_R\to \kq_R^\cell \to \C_R^\cell \]
  of cellular motivic spectra over $R$.
  Since $(-)^\cell$ preserves connectivity,
  the homotopy group connectivity of $\C_R^\cell$ coincides with
  its $2$-connectivity.
  This implies that $\pi_{1+(\star)}\C_R^\cell=0$, whence
  $\pi_{1+(\star)}\unit_R\to \pi_{1+(\star)}\kq_R^\cell\iso \pi_{1+(\star)}\kq_R$
  is surjective and
  $\pi_{0+(\star)}\unit_R\to \pi_{0+(\star)}\kq_R^\cell\iso \pi_{0+(\star)}\kq_R$
  is an isomorphism.
  The identification with naive Milnor-Witt $K$-theory follows
  from Theorem~\ref{thm:pi0kq-integers}.
\end{proof}

\begin{remark}\label{rem:pi1-sphere-z}
  Identifying the kernel of
  the surjection $\pi_{1-(\star)}\unit_\Z\to \pi_{1-(\star)}\kq_\Z$
  with $\KMil_{2+\star}(\Z)/24$, just as in \cite{rso.oneline, rso.twoline},
  but without inverting any prime number, requires more work.
  If the slices of $\unit_\Z$ have the form conjectured by
  Voevodsky in\cite{voevodsky.open} (which
  is the case for the first three nontrivial
  slices~\cite[Lemma 4.5, Lemma 4.7]{kro}),
  a suitable modification of
  the slice convergence theorem \cite[Theorem 3.50]{rso.oneline}
  would enable one to lift the computation.
  This would also involve a discussion of
  $\pi_{s+(\star)}\unit_\Z[\eta^{-1}]$ for small $s$.
  A noteable difference to the vanishing
  $\pi_1\unit_F[\eta^{-1}]=0$ proved in
  \cite{r.etainv} over fields of characteristic not $2$ is 
  that $\pi_{1+(\star)}\unit_\Z[\eta^{-1}]\iso \pi_{1+(\star)}\kw_\Z
  \iso \Z/2\Z[\eta^{\pm 1}]$ by Corollary~\ref{cor:sphere-etainv-kw-conn} below
  and \cite[Theorem 10.13]{kro}.
  Nevertheless, under the hypothesis that the slices of $\unit_\Z$
  have the expected form, one may
  remove inversion of the characteristic in the
  determination of the first Milnor-Witt stem of the sphere spectrum
  over fields of odd characteristic
  \cite[Theorem 5.5]{rso.oneline}, \cite[Theorem 4.1]{rso.twoline},
  as well as in the version for fields of characteristic $2$ given below in
  Theorem~\ref{thm:pi1-sphere-char2}.
\end{remark}

\begin{remark}\label{rem:refined-enum-geom}
  Theorem~\ref{thm:pi0-sphere-z} implies that any refined enumerative
  counting problem which can be phrased over $\Z$ has a degree in
  $\pi_{0+(0)}\unit_\Z\iso \KMW_0(\Z)\iso \Z\oplus \Z$. Hence
  such a degree
  is the class of a symmetric bilinear form uniquely expressible as
  a linear combination of $\langle 1\rangle$ and the hyperbolic plane $\hyp$.
\end{remark}

The connectivity of the unit map
$\unit\to \kq$ improves after inverting $\eta$ as follows.

\begin{corollary}\label{cor:sphere-etainv-kw-conn}
  Let $S$ be an essentially smooth scheme over a Dedekind scheme $\mathcal{D}$.
  The cofiber of the unit map $\unit_S[\eta^{-1}]\to \kw_S$ is $4$-connective.
\end{corollary}

\begin{proof}
  Theorem~\ref{thm:sphere-kq-eff} implies that the cofiber $\C_S$ is
  at least $2$-connective, since inverting $\eta$ preserves connectivity.
  To obtain the higher connectivity, it suffices to prove this over a field
  $F$. For fields of characteristic not $2$, \cite[Theorem 8.3]{r.etainv}
  and the known vanishing $\pis_s\kw_F=0$ for $s$ not divisible by $4$
  implies the desired connectivity. Suppose now that $F$ has
  characteristic $2$. The right unit $\kw_F\hookrightarrow \kw_F\smash \kw_F$
  is the inclusion of the lowest summand in
  the decomposition of the target given in Proposition~\ref{prop:kw-coop-char2}.
  It follows that $\C_F\smash \kw_F$ is $4$-connective. The
  cofiber sequence
  $\C_F\to \C_F\smash \kw_F\to \C_F\smash \C_F \to \Sigma \C_F$,
  in which $\C_F\smash \C_F$ is $4$-connective,
  implies that $\C_F$ is at least $3$-connective. But then $\C_F\smash \C_F$
  becomes $6$-connective, giving that $\C_F$ has the same connectivity
  as $\C_F\smash \kw_F$.
\end{proof}

As a consequence of the connectivity of the unit map
$\unit[\eta^{-1}]\to \kw$, the $\eta$-inverted sphere
$\unit[\eta^{-1}]$ may be described via an Adams-Novikov resolution
based on $\kw$.

\begin{lemma}\label{lem:sphere-etainv-kwresolution}
  Let $S$ be an essentially smooth scheme over a Dedekind scheme $\mathcal{D}$.
  Let $\unit[\eta^{-1}]\to \kw$ denote the unit map.
  Then the canonical map
  \[\unit[\eta^{-1}]\to \holim_{[m]\in \Delta}\kw^{\smash m+1}\]
  is an equivalence.
\end{lemma}

\begin{proof}
  Observe first that the map is sufficiently natural in the sense
  that for any motivic spectrum $\E$, there is a canonical map
  \[\phi_\E\colon \unit[\eta^{-1}]\smash \E\to
    \holim_{[m]\in \Delta}\kw^{\smash m+1}\smash \E \]
  natural in $\E$. If $\E$ is a $\kw$-module, $\phi_\E$ is an equivalence,
  since in that case the homotopy limit identifies with $\E$ via the action
  map $\kw\smash \E\to \E$. Let $\C$ denote the cofiber of
  $\unit[\eta^{-1}]\to \kw$, which is $4$-connective by
  Corollary~\ref{cor:sphere-etainv-kw-conn}. Then $\phi_{\C^{\smash \ell}}$
  induces an isomorphism of
  (trivial) homotopy sheaves $\pis_{s+(\star)}$ for $s<4\ell -2$.
  The cofiber sequence
  \[ \C^{\smash \ell} \to \kw \smash \C^{\smash \ell}\to \C^{\smash \ell +1}
    \to \Sigma C^{\smash \ell}\]
  involving a free $\kw$-module implies that the upper bound may be
  increased to $4\ell +2$. But then also for $\C^{\smash \ell +1}$,
  thus even to $4\ell +4$ for $\C^{\smash \ell}$ and so on,
  giving the statement.
\end{proof}

The $\kw$-resolution of $\unit[\eta^{-1}]$ has two applications,
a homotopy sheaf computation and a purity statement.
In the statement of the next theorem, the third algebraic Hopf
map $\sigma\in \pi_{3+(4)}\unit$ appears.

\begin{theorem}\label{thm:sphere-etainv-char2}
  Let $F$ be a field of characteristic $2$. Then
  there exists a unique lift $\tilde{\alpha}\in \pi_{2+(2)}\unit_F[\eta^{-1}]$
  of $\alpha\in \pi_{2+(2)}\kw_F$, giving an isomorphism
  \[ \pis_{\ast+(\star)}\unit_F[\eta^{-1}] \iso
    \Witt[\eta^{\pm 1},\tilde{\alpha},\sigma]/(\sigma^2).\]
\end{theorem}

\begin{proof}
  The connectivity in Corollary~\ref{cor:sphere-etainv-kw-conn} shows
  the unique lifting of $\alpha$ to $\tilde{\alpha}$.
  It suffices to provide the isomorphism on global sections of 
  $\pi_{\ast+(\star)}$. The Witt ring of a field of characteristic $2$
  is a local ring, with maximal ideal the fundamental one. Thereby
  it suffices to consider the case of a field with Witt ring $\Z/2\Z$.
  The $\kw$-resolution given in Lemma~\ref{lem:sphere-etainv-kwresolution}
  determines $\unit[\eta^{-1}]$ in general, and over the complex numbers
  with Witt ring $\Z/2\Z$ gives the answer
  \[ \pi_{\ast+(\star)}\unit_\CC[\eta^{-1}] \iso
    \Witt(\CC)[\eta^{\pm 1},\tilde{\beta^4},\sigma]/(\sigma^2)\]
  by \cite{andrews-miller,guillou-isaksen.etainv-complex}.
  Here $\tilde{\beta^4}$ is the unique lift of the corresponding
  polynomial generator for $\pi_{\ast+(\star)}\kw_\CC$.
  The result follows.
\end{proof}

One consequence is that the identification of the first Milnor-Witt stem
of the motivic sphere spectrum now works over fields of characteristic $2$.

\begin{theorem}\label{thm:pi1-sphere-char2}
  Let $F$ be a field of characteristic $2$. Then the unit map
  $\pi_1\unit_F\to \pi_1\kq_F$ is surjective, and its kernel
  coincides with $\KMil_{\star+2}/12$ generated by the second
  algebraic Hopf map after inverting~$2$.
\end{theorem}

\begin{proof}
  The surjectivity follows from the
  connectivity statement in Theorem~\ref{thm:sphere-kq-eff}.
  For the identification of the kernel, the slice spectral
  sequence may be used. Note that the form of the slices of
  $\unit_F$ is only known after inverting $2$;
  see \cite{levine.comparison, rso.oneline}. In that case,
  the column of the slice spectral sequence for $\unit_F[2^{-1}]$
  responsible for $\pi_{1}$ consists of a single nonzero
  entry, coming from the $2$-slice
  \[ \s_2\unit_F[2^{-1}] \simeq \Sigma^{1+(2)}\HZ_F/12[2^{-1}]\simeq
    \Sigma^{1+(2)}\HZ_F/3.\]
  No nontrivial slice
  differential can originate from this term.
  Potentially a second slice differential $\pi_2\HZ_F[2^{-1}]\to
  \pi_1\Sigma^{1+(2)}\HZ_F/3$ could hit it.
  Passage from $\unit_F$ to $\unit_F/3$ leads to a commutative
  square of second slice differentials whose source now is
  $\pi_2\HZ_F/3$. This $\KMil$-module is generated by
  an element $g_F\in \pi_{2-(2)}\HZ_F/3 \iso \Z/3\Z$, which lives
  over the prime field $\FF_2$. The image of $g_{\FF_2}$
  lies in the zero group $\pi_1\Sigma^{1+(2)}\HZ_{\FF_2}/3=0$, showing
  its triviality. There results an identification
  $\pi_{1+(\star)}(\unit_F[2^{-1}])^\wedge_\eta \iso \pi_{0+(\star-2)}\HZ_F/3$
  using the convergence statement \cite[Theorem 3.50]{rso.oneline}.
  It remains to observe that $\unit_F[2^{-1},\eta^{-1}]\simeq \ast$
  to conclude the isomorphism
  $\pi_{1+(\star)}\unit_F[2^{-1}]\iso \pi_{1+(\star)}(\unit_F[2^{-1}])^\wedge_\eta$.
  Note that $\pi_1\unit_F[\eta^{-1}]=0$
  by Theorem~\ref{thm:sphere-etainv-char2} already before inverting $2$.
\end{proof}

\begin{remark}\label{rem:etatopchar2}
  The form of Theorem~\ref{thm:pi1-sphere-char2} indicates
  that over a field $F$ of characteristic $2$, the
  topological first Hopf element $\eta_\Top=0\in \pi_{1+(0)} \unit_F$.
  Indeed, the unique nonzero element representing $\eta_\Top$ in
  $\pi_{1+(0)}\s_1\unit$ over $\Z$ (which is ``the same''
  over all fields of characteristic
  not $2$, namely the primitive second root of unity $-1$)
  becomes $0$ over fields of characteristic $2$.
  As another indicator,
  suppose the Hopf relation $12\nu =\eta^2\eta_\Top$ holds
  over $\Z$. It can be established using the
  multiplicative structure on the known slices of $\unit_\Z$,
  up to an element in
  the image of $\pi_{1+(2)}\f_4\unit_\Z\to \pi_{1+(2)}\unit_Z$.
  Ananyevskiy observed that $\nu\in \pi_{1+(2)}\unit_F$ has order
  at most $12$ already unstably.
  Note that $\eta_\Top=0\in \pi_{1+(0)}\kq_F\iso \pi_{1+(0)}\HZ_F$
  by Proposition~\ref{prop:pi1kqchar2}.
\end{remark}

\begin{corollary}\label{cor:absolute-purity-etainv}
  Let $S$ be a connected Dedekind scheme whose fraction field has
  characteristic $0$.
  Let $i\colon Z\hookrightarrow S$ be the closed embedding of a point of
  positive characteristic $p$. Then the purity transformation
  $\Sigma^{-1-(1)}i^\ast\unit_S[\eta^{-1}]\to i^!\unit_S[\eta^{-1}]$
  is an equivalence.
\end{corollary}

\begin{proof}
  It suffices to check that the purity transformation is an equivalence
  after inverting $2$ and after localizing at $2$.
  In the first case, the purity transformation has contractible source and
  target, because $\pi_{0+(0)}\unit_F[\eta^{-1}]$,
  the Witt ring of a field $F$ of positive characteristic,
  becomes zero after inverting $2$. The purity statement for $\kq$ given in
  Proposition~\ref{prop:absolute-purity-kq} implies the
  same for $\kw$, and also for $\kw_{(2)}$, since both source and target
  of the purity transformation commute with colimits.
  The $\kw$-resolution of $\unit[\eta^{-1}]$ given in
  Lemma~\ref{lem:sphere-etainv-kwresolution} induces a $\kw_{(2)}$ resolution
  of $\unit[\eta^{-1}]_{(2)}$.
  Since $\kw\smash \kw_{(2)}$ is a direct sum of suspensions of copies of
  $\kw_{(2)}$ by Corollary~\ref{cor:kw-coop}, the purity transformation
  is an equivalence for every $m$-fold smash power of $\kw_{(2)}$.
  There results an equivalence
  \begin{align*}
    \unit_Z[\eta^{-1}]_{(2)} & \simeq
      \holim_{[m]\in \Delta} (\kw_Z)_{(2)}^{\smash m+1} \\
                             & \simeq
      \holim_{[m]\in \Delta} i^\ast((\kw_S)_{(2)}^{\smash m+1}) \\
                             & \simeq
      \holim_{[m]\in \Delta} \Sigma^{1+(1)}i^!((\kw_S)_{(2)}^{\smash m+1}) \\
                             &\simeq
     \Sigma^{1+(1)}i^!\bigl(\holim_{[m]\in \Delta} (\kw_S)_{(2)}^{\smash m+1})\\
                             & \simeq
                               \Sigma^{1+(1)}i^!\unit_S[\eta^{-1}]_{(2)}
  \end{align*}
  which, being compatible with the purity transformation, completes the proof.
\end{proof}

\begin{remark}\label{rem:ishriek-cellular}
  As an immediate consequence of Corollary~\ref{cor:absolute-purity-etainv},
  $i^!$ sends cellular motivic spectra in $\SH(S)[\eta^{-1}]$ to cellular
  motivic spectra in $\SH(Z)[\eta^{-1}]$ under
  the assumption on $i$.
\end{remark}

\newcommand{\etalchar}[1]{$^{#1}$}

\end{document}